\documentclass[noinfoline]{imsart}

\usepackage{amsmath}

\usepackage{graphicx}
\usepackage{color}
\usepackage{amsmath}
\usepackage{amssymb}
\usepackage{amscd}
\usepackage{bbm}

\usepackage{tikz}
\newlength\Colsep
\usetikzlibrary{patterns}

\usepackage{hyperref}
\hypersetup{
	bookmarks=true,         
	colorlinks=true,       
	linkcolor=red,          
	citecolor=green,        
	filecolor=magenta,      
	urlcolor=cyan           
}

\usepackage[utf8]{inputenc}
\usepackage{amsthm}
\usepackage{bbm} 
\usepackage[linesnumbered,lined,boxed,commentsnumbered]{algorithm2e}

\usepackage{marginnote}

\newtheorem{Theorem}{Theorem}
\newtheorem{Assumption}{Assumption}
\newtheorem{Definition}{Definition}
\newtheorem{Lemma}{Lemma}
\newtheorem{Proposition}{Proposition}
\newtheorem{Corollary}{Corollary}
\newtheorem{Remark}{Remark}

\newcommand{\inr}[1]{\bigl< #1 \bigr>}

\newcommand{\norm}[1]{\left\|#1\right\|}%

\DeclareMathOperator*{\argmin}{argmin}

\def\ds1{\textrm{1\kern-0.25emI}} 

\newcommand \E{\mathbb{E}}
\newcommand \R{\mathbb{R}}

\newcommand \cA{{\cal A}}

\newcommand \cD{{\cal D}}

\newcommand \cF{{\cal F}}

\newcommand \cH{{\cal H}}
\newcommand \cI{{\cal I}}

\newcommand \cL{{\cal L}}

\newcommand \cN{{\cal N}}
\newcommand \cO{{\cal O}}
\newcommand \cP{{\cal P}}

\newcommand \cX{{\cal X}}
\newcommand \cY{{\cal Y}}

\newcommand \bB{{\mathbb B}}

\newcommand \bE{{\mathbb E}}

\newcommand \bI{{\mathbb I}}

\newcommand \bN{{\mathbb N}}

\newcommand \bP{{\mathbb P}}

\newcommand \bR{{\mathbb R}}
\newcommand \bS{{\mathbb S}}

\usepackage{tcolorbox}

\begin{document}

\begin{frontmatter}

\title{ERM and RERM are optimal estimators for regression problems when malicious outliers corrupt the labels}
\runtitle{RERM with malicious outliers corrupting the labels}

\begin{aug}
  \author{\fnms{CHINOT}  \snm{Geoffrey} \ead[label=e1]{geoffrey.chinot@ensae.fr}}

  \runauthor{CHINOT Geoffrey}

  \affiliation{ENSAE, CREST and Institut Polytechnique de Paris}

\address{ENSAE, 5 avenue Henri Chatelier, 91120, Palaiseau, France\\ 
	\printead{e1}}

\end{aug}

\begin{abstract}
	We study Empirical Risk Minimizers (ERM) and  Regularized Empirical Risk Minimizers (RERM) for regression problems with convex and $L$-Lipschitz loss functions. We consider a setting where $|\cO|$ malicious outliers contaminate the labels. In that case, under a local Bernstein condition, we show that the $L_2$-error rate is bounded by $ r_N + AL |\cO|/N$, where $N$ is the total number of observations, $r_N$ is the $L_2$-error rate in the non-contaminated setting and $A$ is a parameter coming from the local Bernstein condition. When $r_N$ is minimax-rate-optimal in a non-contaminated setting, the rate $r_N + AL|\cO|/N$ is also minimax-rate-optimal when $|\cO|$ outliers contaminate the label. The main results of the paper can be used for many non-regularized and regularized procedures under weak assumptions on the noise. We present results for Huber's M-estimators (without penalization or regularized by the $\ell_1$-norm) and for general regularized learning problems in reproducible kernel Hilbert spaces when the noise can be heavy-tailed.  
\end{abstract}



\end{frontmatter}

\section{Introduction} \label{sec:intro}

Let $(\Omega, \cA, P)$ be a probability space where $\Omega = \cX \times \cY$. $\cX$ denotes the measurable space of the inputs and $\cY \subset \bR$ the measurable space of the outputs. Let $(X,Y)$ be a random variable taking values in $\Omega$ with joint distribution $P$ and let $\mu$ be the marginal distribution of $X$. Let $F$ denote a class of functions $f: \cX \mapsto \cY$. A function $f$ in $F$ is named a \textit{predictor}. The function $\ell : \cY \times \cY \mapsto \bR^+$ is a loss function such that $\ell(f(x),y)$ measures the quality of predicting $f(x)$ while the true answer is $y$. For any function $f$ in $F$ we write $\ell_f(x,y) := \ell(f(x),y) $. For any distribution $Q$ on $\Omega$ and any function $f: \cX \times \cY \mapsto \bR$ we write $Qf = \bE_{(X,Y) \sim Q} [f(X,Y)]$. Let $f \in F$, the risk of $f$ is defined as $R(f) := P \ell_f = \bE_{(X,Y) \sim P} [\ell(f(X),Y)]$. A prediction function with minimal risk is called an \textit{oracle} and is defined as $f^* \in \argmin_{f \in F} P\ell_f$. For the sake of simplicity, it is assumed that the \textit{oracle} $f^*$ exists and is unique. The joint distribution $P$ of $(X,Y)$ being unknown, computing $f^*$ is impossible. Instead one is given a dataset $\cD = (X_i,Y_i)_{i=1}^N$ of $N$ random variables taking values in $\cX \times \cY$. \\

Regularized empirical risk minimization is the most widespread strategy in machine learning to estimate $f^*$. There exists an extensive literature on its generalization capabilities \cite{vapnik1998statistical,koltchinskii2006local,MR2829871,LM_sparsity,ChiLecLer:2018}. However, in the past few years, many papers have highlighted its severe limitations. One main drawback is that a single outlier $(X_o,Y_o)$ (in the sense that nothing is assumed on $(X_o,Y_o)$) can deteriorate the performances of RERM. Consequently, RERM is in general not robust to outliers. However, the question below naturally follows: 

\begin{center}
	What happens if only the labels $(Y_i)_{i=1}^N$ are contaminated ? 
\end{center}
For example, in~\cite{dalalyan2019outlier}, the authors raised the question whether it is possible to attain optimal rates of convergence in outlier-robust sparse regression using regularized empirical risk minimization. They consider the model, $Y_i = \inr{X_i,t^*} + \epsilon_i$, where $X_i$ is a Gaussian random vector in $\bR^p$ with a covariance matrix satisfying the Restricted Eigenvalue condition \cite{van2009conditions} and $t^*$ is assumed to be $s$-sparse. The non-contaminated noise is $\epsilon_i \sim \cN(0,\sigma^2)$, while it can be anything when malicious outliers contaminate the labels. The authors prove that the $\ell_1$-penalized empirical risk minimizer based on the Huber's loss function has an error rate of the order
\begin{equation} \label{rate_arnak}
\sigma \bigg (\sqrt{s\frac{\log(p/\delta)}{N}} + \frac{|\cO| \log(N/\delta)}{N} \bigg) \enspace,
\end{equation}
with probability larger that $1-\delta$, where $|\cO|$ is the number of outliers. Up to a logarithmic factor, the RERM associated with the Huber loss function for the problem or sparse-linear regression is minimax-rate-optimal when $|\cO|$ malicious outliers corrupt the labels.

\subsection{Setting}
In this paper, we consider a setup where $|\cO|$ outputs can be contaminated. More precisely, let $\cI \cup \cO$ denote an unknown partition of $ \{1,\cdots,N\}$ where $\cI$ is the set of \textbf{informative} data and $\cO $ is the set of \textbf{outliers}.
\begin{Assumption}\label{assum:distri}
	$(X_i,Y_i)_{i\in \cI}$ are i.i.d with a common distribution $P$. The random variables $(X_i)_{i =1}^N$ are i.i.d with law $\mu$.
\end{Assumption}
Nothing is assumed on the labels $(Y_i)_{i \in \cO}$. They can be adversarial outliers making the learning as hard as possible. Without knowing the partition $\cI \cup \cO$, the goal is to construct an estimator $\hat f_N$ that approximates/estimates the \textit{oracle} $f^*$. A way of measuring the quality of an estimator is via the \textbf{error rate} $\|\hat f_N-f\|_{L_2(\mu)}$ or \textbf{the excess risk} $P\cL_{\hat{f}} := P\ell_{\hat{f}_N} - P \ell_{f^*}$. 
We assume the following:
\begin{Assumption}\label{assum:convex}
	The class $F$ is convex.
\end{Assumption}
A natural idea to construct robust estimators when the labels might be contaminated is to consider $L$-Lipschitz loss functions~\cite{huber1992robust,huber2011robust} that is a losses satisfying 
\[
\big | \ell \big(  f(x),y \big) - \ell \big(  g(x),y \big) \big| \leq L | f(x)- g(x)| \enspace,
\]
for every $x,y \in \cX \times \cY $ and $f,g \in F$. Moreover, for computational purposes we also focus on convex loss functions~\cite{van2016estimation}. 
\begin{Assumption}\label{assum:lip_conv}
	There exists $L>0$ such that, for any $y \in \cY$, ${\ell}(\cdot,y)$ is $L$-Lipschitz and convex.
\end{Assumption}
Recall that the Empirical Risk Minimizer (ERM) and the Regularized Empirical Risk Minimizer (RERM) are respectively defined as
\begin{equation*}
\hat{f}_N \in \argmin_{f \in F} \frac{1}{N} \sum_{i=1}^N \ell(f(X_i),Y_i) \quad \mbox{and} \quad  \hat{f}_N^{\lambda} \in \argmin_{f \in F} \frac{1}{N} \sum_{i=1}^N \ell(f(X_i),Y_i)  + \lambda \|f\| \enspace,
\end{equation*}
where $\lambda >0$ is a tuning parameter and $\|\cdot\|$ is a norm. Under Assumptions~\ref{assum:convex} and~\ref{assum:lip_conv} the ERM and RERM are computable using tools from convex optimization~\cite{boyd2004convex}.\\

\subsection{Our contributions}

As exposed in~\cite{dalalyan2019outlier}, in a setting where $|\cO|$ outliers contaminate only the labels, RERM with the Huber loss function is (nearly) minimax-rate-optimal for the sparse-regression problem when the noise and design of non-contaminated data are both Gaussian. It leads to the following questions:

\begin{enumerate}
	\item 	\textit{Is the (R)ERM optimal for other loss functions and regression problems than the sparse-regression when malicious outliers corrupt the labels ?}
	\item \textit{Is the Gaussian assumption on the noise necessary ?}
\end{enumerate}
Based on the previous works~\cite{ChiLecLer:2018,chinot2019robust,ChiLecLer:2019,alquier2017estimation}, we study both ERM and RERM for regression problems when the penalization is a norm and the loss function is simultaneously convex and Lipschitz (Assumption~\ref{assum:lip_conv}) and show that:
\begin{tcolorbox}
	In a framework where  $|\cO|$ outliers may contaminate the labels, under a local Bernstein condition, the error rate for both ERM and RERM can be bounded by
	\begin{equation} \label{opt_bound}
	r_N +  AL  \frac{|\cO|}{N} \enspace,
	\end{equation}
	where $N$ is the total number of observations, $L$ is the Lipschitz constant from Assumption~\ref{assum:lip_conv}, $r_N$ is the error rate in a non-contaminated setting and $A$ is a constant coming from the local Bernstein condition.
\end{tcolorbox}
When the proportion of outliers $|\cO|/N$ is smaller than the error rate $r_N/(AL)$, both ERM and RERM behave as if there was no contamination. The result holds for any loss function that is simultaneously convex and Lipschitz and not only for the Huber loss function. We obtain theorems that can be used for many well-known regression problems including structured high-dimensional regression (see Section~\ref{app:huber_regu}), non-parametric regression (see Section~\ref{app:rkhs}) and  matrix trace regression (using the results from~\cite{alquier2017estimation}). As a proof of concept, for the problem of sparse-linear regression, we improve the rate~\eqref{rate_arnak}, even when the noise may be heavy-tailed. The next question one may ask is the following:
\begin{enumerate} \setcounter{enumi}{1}
	\item 	\textit{ Is the general bound~\eqref{opt_bound} minimax-rate-optimal when $|\cO|$ malicious outliers may corrupt the labels ?}
\end{enumerate}
 To answer question 2, we use the results from~\cite{chen2018robust}. The authors established a general minimax theory for the $\varepsilon$-contamination model defined as $P_{(\varepsilon,\theta,Q)} = (1-\varepsilon)P_{\theta} + \varepsilon Q$ given a general statistical experiment $\{P_{\theta}, \theta \in \Theta \}$. A deterministic proportion $\varepsilon$ of outliers with same the distribution $Q$ contaminates $P_{\theta}$. When $Y = f_{\theta}(X) + \epsilon$, $\theta \in \Theta$  and following the idea of~\cite{chen2018robust}, we show in Section~\ref{sec:minimax} that the lower minimax bounds for regression problems in the $\varepsilon$-contamination model are the same when
\begin{enumerate}
	\setlength{\itemsep}{1pt}
	\setlength{\parskip}{0pt}
	\setlength{\parsep}{0pt}	
	\item  Both the design $X$ and the response variable $Y$ are contaminated.
	\item Only the response variable $Y$ is contaminated. 
\end{enumerate}
Since in our setting, outliers do not necessarily have the same distribution $Q$, it is clear that a lower bound on the risk in the $\varepsilon$-contamination model implies a lower bound when $|\cO| = \varepsilon N$ arbitrary outliers contaminate the dataset 
As a consequence, for regression problems, minimax-rate-optimal bounds in the $\varepsilon$-contamination model are also optimal when $N\varepsilon$ malicious outliers corrupt only the labels.
\begin{tcolorbox}
	When the bound~\eqref{opt_bound} is minimax-rate-optimal for regression problems in the $\varepsilon$-contamination model with $\varepsilon= |\cO|/N$, then it is also minimax-rate-optimal when $|\cO|$ malicious outliers corrupt the labels. 
\end{tcolorbox}
The results are derived under the local Bernstein condition introduced in~\cite{ChiLecLer:2018}. This condition enables to obtain fast rates of convergence, even when the noise is heavy-tailed. As a proof of concept, we study Huber's $M$-estimators in $\bR^p$ (non-penalized or regularized by the $\ell_1$-norm) when the noise is Cauchy. In these cases, the error rates are respectively
\[
L \bigg(\sqrt{\frac{\textnormal{Tr}(\Sigma)}{N}} + \frac{|\cO|}{N} \bigg) \quad \textnormal{and} \quad  L \bigg( \sqrt{\frac{s\log(p)}{N}} + \frac{|\cO|}{N} \bigg) \enspace,
\]
where $\Sigma$ is the covariance matrix of the design $X$. We also study learning problems in general Reproducible Kernel Hilbert Space (RKHS). We derive error rates depending on the spectrum of the integral operator as in  \cite{smale2007learning,mendelson2010regularization,caponnetto2007optimal} without assumption on the design and when the noise is heavy-tailed (see section~\ref{app:huber_regu}). \\

\begin{Remark}
	The general results hold for any Lipschitz and convex loss function. However, for the sake of simplicity, we present applications only for the Huber loss function. Using results from~\cite{ChiLecLer:2018}, it is possible to apply our main results to other Lipschitz loss functions. 
\end{Remark}

\subsection{Related Literature}
Regression problems with possibly heavy-tailed data or outliers cannot be handled by classical least-squares estimators. This lack of robustness of least-squares estimators gave birth to the theory of robust statistics developed by Peter Huber~\cite{huber1992robust,huber2011robust,huber1967behavior}, John Tukey~\cite{tukey1960survey,tukey1962future} and Frank Hampel~\cite{hampel1971general,hampel1974influence}. The most classical alternatives to least-squares estimators are M-estimators. They consist in replacing the quadratic loss function by other loss functions, less sensitive to outliers~\cite{maronna1976robust,yohai1979asymptotic}. \\
Robust statistics has attracted a lot of attention in the past few years both in the computer science and the statistical communities. 
For example, although estimating the mean of a random vector in $\bR^p$ is one of the oldest and fundamental problems in statistics, it is still a very active research area. Surprisingly, optimal bounds for heavy-tailed data have been obtained only recently~\cite{lugosi2019sub}. However, their estimator cannot be computed in practice. Using semi-definite programming (SDP),~\cite{hopkins2018sub} obtained optimal bounds achievable in polynomial time. In a recent works, still using SDP,~\cite{lecue2019robust} designed an algorithm computable in nearly linear time, while~\cite{lei2019fast} developed the first tractable optimal algorithm not based on the SDP. 
\\
In the meantime, another recent trend in robust statistics has been to focus on finite sample risk bounds that are minimax-rate-optimal when $|\cO|$ outliers contaminate the dataset. For example, for the problem of mean estimation, when $|\cO|$ malicious outliers contaminate the dataset and the non-contaminated data are assumed to be sub-Gaussian, the optimal rate (measured in Euclidean norm) scales as $\sqrt{p/N} + |\cO|/N$.  In~\cite{chen2018robust}, the authors developed a general analysis for the $\varepsilon$-contamination model. In~\cite{chen2016general}, the same authors proposed an optimal estimator when $|\cO|$ outliers with the same distribution contaminate the data. In \cite{diakonikolas2019efficient}, the authors focused on the problem of high-dimensional linear regression in a robust model where an $\varepsilon$-fraction of the samples can be adversarially corrupted. Robust regression problems have also been studied in~\cite{cheng2019high,diakonikolas2019robust,liu2018high,bhatia2015robust}. Above-mentioned articles assume corruption both in the design and the label. In such a corruption setting ERM and RERM are known to be poor estimators. \\

In~\cite{dalalyan2019outlier}, the authors raised the question whether it is possible to attain optimal rates of convergence in sparse regression using regularized empirical risk minimization when a proportion of malicious outliers contaminate only the labels. They studied $\ell_1$ penalized Huber's $M$-estimators. This work is the closest to our setting and reveals that when only the labels are contaminated, simple procedures, such as penalized Huber's $M$ estimators, still perform well and are minimax-rate-optimal. Their proofs rely on the fact that non-contaminated data are Gaussian. Our approach is different, more general and uses the control of stochastic processes indexed by the class $F$. \\

Other alternatives to be robust both for heavy-tailed data and outliers in regression have been proposed in the literature such as Median Of Means (MOM) based methods \cite{lecue2017robust,lecue2018robust,ChiLecLer:2018}. However such estimators are difficult to compute in practice and can lead to sub-optimal rates. For instance, for sparse-linear regressions in $\bR^p$ with a sub-Gaussian design, MOM-based estimators have an error rate of the order $L\big(\sqrt{s\log(p)/N} +  \sqrt{|\cO|/N}\big)$ (see~\cite{ChiLecLer:2018}) while the optimal dependence with respect to the number of outliers is $L \big(\sqrt{s\log(p)/N} + |\cO| /N \big)$. Finally, there was a recent interest in robust iterative algorithms. It was shown that robustness of stochastic approximation algorithms can be enhanced by using robust stochastic gradients. For example, based on the geometric median~\cite{minsker2015geometric},~\cite{chen2017distributed} designed a robust gradient descent scheme. More recently,~\cite{juditsky2019algorithms} showed that a simple truncation of the gradient enhances the robustness of the stochastic mirror descent algorithm. \\

The paper is organized as follows. In Section~\ref{sec:no_reg}, we present general results for non-regularized procedures with a focus on the example of the Huber's $M$-estimator in $\bR^p$. Section~\ref{sec:high_dim} gives general results for RERM that we apply to $\ell_1$-penalized Huber's $M$-estimators with isotropic design and regularized learning in RKHS. Section~\ref{sec:simu} presents simple simulations to illustrate our theoretical findings. In section~\ref{sec:minimax}, we show that the minimax lower bounds for regression problems in the $\varepsilon$-contamination model are the same when 1) both the design $X$ and the labels are contaminated and 2) when only the labels are contaminated. Section~\ref{RE_huber} shows that we can extend the results for $\ell_1$-penalized Huber's $M$-estimator when the covariance matrix of the design $X$ satisfies a Restricted Eigenvalue condition. Finally, the proofs of the main theorems are presented in Section~\ref{sec:proof_main_thm}.

\paragraph{Notations} 
All along the paper, for any $f$ in $F$, $\|f\|_{L_2}$ will be written instead of $\|f\|_{L_2(\mu)}  =\int f^2 d \mu $. We also write $L_2$ instead of $L_2(\mu)$. Let $\bB_2$ and $\bS_2$ be respectively the unit ball and the unit sphere with respect to the metric $L_2(\mu)$. The letter $c$ will denote an absolute constant whose value may change from one line to another. For a set $T$, its cardinality is denoted $|T|$. For two real numbers $a,b$, $ a \vee b$ and $a \wedge b$ denote respectively $\max(a,b)$ and $\min(a,b)$. For any set $H$ for which it makes sense, let $H + f^{*} = \{ h+f^{*}, h \in H   \}$, $H - f^{*} = \{ h-f^{*}, h \in H   \}$.

\section{Non-regularized procedures} \label{sec:no_reg}
In this section, we study the Empirical Risk Minimizer (ERM) where we recall the definition below:
\begin{equation} \label{def:estimator}
\hat{f}_N = \arg \min_{f \in F} \frac{1}{N} \sum_{i=1}^{N} \ell(f(X_i),Y_i) \enspace.
\end{equation}
We establish bounds on  the error rate $\| \hat{f}_N - f^*\|_{L_2}$ and the excess risk $P\cL_{\hat{f}_N} := P\ell_{\hat{f}_N} - P\ell_{f^*}$ in two different settings 1) when $F-f^*$ is sub-Gaussian, and 2) when $F-f^*$ is locally bounded. We derive fast rates of convergence under very weak assumptions. 

\subsection{Complexity measures and parameters}
The ERM performs well when the empirical excess risk $f \mapsto P_N \cL_f$ uniformly concentrates around its expectation $f \mapsto P\cL_f$. From Assumption~\ref{assum:lip_conv}, such uniform deviation results depend on the complexity of the class $F$. There exist different measures of complexity. In this section, we introduce the two main complexity measures we will use throughout this article. Let $H \subseteq F \subseteq L_2$.
\begin{enumerate}
	\item  Let $(G_h)_{h\in H}$ be the centered Gaussian process indexed by $H$ where the covariance structure of $(G_h)_{h\in H}$ is given by  $ \E (G_{h_1}- G_{h_2})^2 =  \E(h_1(X)-h_2(X))^2$ for all $h_1,h_2\in H$. The \textbf{Gaussian mean-width} of $H$ is defined as
	\begin{equation} \label{def_gaussian_mean_width}
	w(H) = \E \sup_{h\in H} G_h \enspace. 
	\end{equation}
	For example, for $\mu = \cN(0,\Sigma)$, $T \subset \bR^p$ and $F = \{ \inr{t,\cdot}, t \in T  \}$, we have $w(F) = \bE \sup_{t \in T }\Sigma^{1/2} \inr{t, \mathbf{G}}$, where $\mathbf{G} \sim \cN(0,I_p)$.\\
	
	\item Let $S \subseteq \{1,\cdots,N\}$ and $(\sigma_i)_{i \in S}$ be i.i.d Rademacher random variables ($P(\sigma_i = 1) = P(\sigma_i = -1) = 1/2$ ) independent to $(X_i)_{i \in S}$. The \textbf{Rademacher complexity} of $H$, indexed by $S$ is defined as
	\begin{equation} \label{def_Rademacher_comp}
	\textnormal{Rad}_{S}(H) = \bE \sup_{h \in H} \sum_{i \in S} \sigma_i h(X_i) \enspace,
	\end{equation}
	where this expectation is taken both with respect to $(X_i)_{i \in S}$ and $(\sigma_i)_{i  \in S}$. The Rademacher complexity has been extensively used in the literature as a measure of complexity~\cite{bartlett2005local,bartlett2006empirical,bousquet2002some}.
\end{enumerate}
Depending on the context, we will use either the Gaussian mean-width or the Rademacher complexity as a measure of complexity. In particular, the Gaussian mean-width naturally appears when dealing with sub-Gaussian classes of functions (see Definition~\ref{def:sub-gauss-class}) while the Rademacher complexity is convenient when dealing with bounded class of functions. 
\begin{Definition} \label{def:sub-gauss-class}
	A class $H \subseteq F \subset L_2$ is called B-\textbf{sub-Gaussian} if for every $\lambda >0$ and $h \in H$
	\begin{equation*}
	\mathbb E \exp( \lambda h(X)/ \|h\|_{L_2} ) \leq \exp(B^2\lambda^2 /2)\enspace.
	\end{equation*} 
\end{Definition}
Now, let us define the two complexity parameters that will drive the rates of convergence of the ERM. 
\begin{Definition}\label{def:comp}
	For any $A>0$, let
	\begin{equation*}
	r_{\cI}^{SG}(A) = \inf \big \{ r>0 :   ALw \big(F\cap  (f^*+ r\bB_2) \big) \leq  c \sqrt{|\cI|} r^2 \big\}
	\end{equation*}
	and 
	\begin{equation*}
	r_{\cI}^B(A) = \inf \bigg\{ r>0 :AL  \textnormal{Rad}_{\cI} \big(F \cap (f^*+ r\bB_2) \big) \leq  c |\cI| r^2 \bigg\}
	\end{equation*}
	where $c>0$ denotes an absolute constant and $L$ is the Lipschitz constant from Assumption~\ref{assum:lip_conv}. Finally, for any $A,\delta >0$ set  
	\begin{equation} \label{eq:comp_par_sg}
	r^{SG}(A,\delta) \geq  c \bigg( r^{SG}_{\cI} (A) \vee AL \sqrt \frac{\log(1/\delta)}{N} \vee AL \frac{|\mathcal O |}{N} \bigg) \enspace,
	\end{equation}
	and
	\begin{equation} \label{eq:comp_par_bounded}
	r^B(A,\delta) \geq  c \bigg( r_{\cI}^B (A) \vee AL \sqrt \frac{\log(1/\delta)}{N} \vee AL \frac{|\mathcal O |}{N} \bigg) \enspace,
	\end{equation}
	where $c>0$ is an absolute constant. 
\end{Definition} 
In Section~\ref{section_local_ber_results} we use $r^{SG}(A,\delta)$ when the class $F-f^*$ is assumed to be sub-Gaussian while we use $r^B(A,\delta)$ when the class $F-f^*$ is (locally) bounded.  

\subsection{Local Bernstein conditions and main results} \label{section_local_ber_results}

To obtain fast rates of convergence, it is necessary to impose assumptions on the distribution $P$. For instance, the margin assumptions \cite{MR1765618,MR2051002,MR3526202} and the Bernstein conditions from \cite{MR2240689} have been widely used in statistics and learning theory. 
A class $F$ is called $(1,A)$ Bernstein~\cite{bartlett2006empirical} if for all $f$ in $F$, $P (\cL_f)^2 \leq A P\cL_f$.  Under Assumption~\ref{assum:lip_conv}, $F$ is $(1,AL^2)$ Bernstein if for all $f$ in $F$, $\|f-f^*\|_{L_2}^2 \leq A P\cL_f$. This condition means that the variance of the problem is not too large. In this paper, we use the second version of the Bernstein condition stated above. Moreover, in the spirit of~\cite{ChiLecLer:2018}, we introduce the (much) weaker \textbf{local Bernstein assumption}. Contrary to the global Bernstein condition, our assumption is required to hold only locally around the \textit{oracle} $f^*$ and not for every $f$ in $F$. As we will see in applications, it allows to consider heavy-tailed noise without deteriorating the convergence rates. 

\begin{Assumption}\label{assum:fast_rates} Let $\delta > 0$ and $r(\cdot,\delta) \in \{ r^{SG}(\cdot,\delta),  r^{B}(\cdot,\delta)   \}$, where $r^{SG}(\cdot,\delta)$ and $r^{B}(\cdot,\delta)$ are respectively defined in Equations~\eqref{eq:comp_par_sg} and~\eqref{eq:comp_par_bounded}. Assume that there exists a constant $A > 0$ such that for all $f\in F \cap \big(f^* + r(A,\delta) \bS_2 \big)$,  we have $\|f-f^{*}\|_{L_2}^2\leq AP\mathcal{L}_f $.
\end{Assumption}
Assumption~\ref{assum:fast_rates} holds locally around the \textit{oracle} $f^*$. The bigger $r(\cdot,\delta)$ the stronger Assumption~\ref{assum:fast_rates}. Assumption~\ref{assum:fast_rates} has been extensively studied in~\cite{ChiLecLer:2018,ChiLecLer:2019} for different Lipschitz and convex loss functions. For the sake of brevity, in applications we will only focus on the Huber loss function in this paper. We are now in position to state the main theorem for the ERM. 
\begin{Theorem} \label{thm:main}
	Let $\cI \cup \cO$ be a partition of $\{1,\cdots,N \}$ where $|\cO| \leq |\cI|$. Grant Assumptions~\ref{assum:distri},~\ref{assum:convex} and~\ref{assum:lip_conv}. Let $\delta \in (0,1)$. 
	\begin{enumerate}
		\item Let us assume that the class $F-f^*$ is $1$-sub-Gaussian and that Assumption~\ref{assum:fast_rates} holds for $r(\cdot, \delta) = r^{SG}(\cdot, \delta)$ and $A >0$. With probability larger than $1-\delta$, the estimator $\hat{f}_N$ defined in Equation~\eqref{def:estimator} satisfies
		\begin{align*}
		\|\hat{f}_N - f^*\|_{{L_2}} \leq r^{SG}(A,\delta)  \quad \mbox{and} \quad   P\cL_{\hat{f}_N} \leq    c\frac{(r^{SG}(A,\delta))^2}{A} \enspace.
		\end{align*}
		\item Let us assume that Assumption~\ref{assum:fast_rates} holds for $r(\cdot, \delta) = r^{B}(\cdot, \delta)$ and $A >0$ and that 
		\begin{equation} \label{cond_locally_bounded}
			\forall f \in F \cap \big( f^* + r^B(A,\delta) \bB_2  \big) \textnormal{ and } x \in \cX \quad |f(x)-f^*(x)| \leq 1	
			\end{equation}
		Then, with probability larger than $1-\delta$, the estimator $\hat{f}_N$ defined in Equation~\eqref{def:estimator} satisfies
		\begin{align*}
		\|\hat{f}_N - f^*\|_{{L_2}} \leq r^B(A,\delta)  \quad \mbox{and} \quad   P\cL_{\hat{f}_N} \leq    c\frac{(r^B(A,\delta))^2}{A}
		\end{align*}
	\end{enumerate}	
\end{Theorem}
The constant $1$ in the sub-Gaussian assumption or in Equation~\eqref{cond_locally_bounded} may be replaced by any other constants. \\

There are two cases in Theorem~\ref{thm:main}:
\begin{enumerate}
	\item When class $F-f^*$ is $1$-sub-Gaussian, the complexity-parameter driving the convergence rates depends on the Gaussian mean-width.
	\item When the class $F-f^*$ is locally bounded (see Equation~\ref{cond_locally_bounded}), the complexity-parameter driving the convergence rates depends on the Rademacher complexity. Equation~\eqref{cond_locally_bounded} requires $L_{\infty}$-boundedness only for functions $f$ in $F \cap (f^* + r^B(A,\delta) \bB_2)$. For example, let $F = \{ \inr{t,\cdot}, t \in \bR^p  \}$ and $X$ be an isotropic random variable, that is $\bE \inr{X,t}^2 = \|t\|^2_2$ for all $t \in \bR^p$. Let $t^*$ be such that $f^*(\cdot) = \inr{t^*,\cdot}$ and $f$ be in $F \cap (f^* + r^B(A,\delta) \bB_2)$. Then,  $|(f-f^*)(x)| = |\inr{t-t^*,x}| \leq \|t-t^*\|_2 \|x\|_2  \leq \|x\|_2  r^B(A,\delta) $. 
	Simple computations (see~\cite{koltchinskii2006local}) show that when $r^B(A,\delta) = r_{\cI}^B(A)$, the complexity parameter $r^B(A,\delta)$ is of the order $\sqrt{p/|\cI|}$ and Equation~\eqref{cond_locally_bounded} holds if
	\[
	\|x\|_2 \leq c \sqrt{|\cI|/p} 
	\]
	The more informative data we have, the larger the euclidean radius of $\cX$ can be.\\
\end{enumerate}
In the case of equality in Equations~\eqref{eq:comp_par_sg} or~\eqref{eq:comp_par_bounded}, Theorem~\ref{thm:main} holds if the local Bernstein condition~\ref{assum:fast_rates} is satisfied for all functions $f$ in $F$ such that:
\[
\|f-f^*\|_{L_2} = c\bigg( r_{\cI}(A) \vee  AL \frac{|\cO|}{N }+ AL \sqrt{\frac{\log (1/\delta)}{N}} \bigg) \enspace,
\]
that is on an $L_2$-sphere around $f^*$ with a radius equal to the rate of convergence. The bound on the error rate can be decomposed as the error rate in the non-contaminated setting and the proportion of outliers $AL|\cO|/N$. As long as $AL |\cO|/N \leq r_{\cI}(A)$, the error rate remains constant and equal the one in a non-contaminating setting. On the other hand, if $AL |\cO|/N \geq r_{\cI}(A)$, the error rate in the contaminated setting becomes linear with respect to the proportion of outliers $|\cO| / N$. Theorem~\ref{thm:main} shows that when $r_{\cI}(A)$ is minimax-rate-optimal in a non-contaminated setting, the ERM remains optimal when less than $N r_{\cI}(A)/(AL)$ outliers contaminate the labels. We also show in Section~\ref{app:hub_reg} that this dependence with respect to the number of outliers is minimax-rate-optimal for linear regression in $\bR^p$ when $|\cO|$ outliers may corrupt the labels.

\subsection{A  concrete example: the class of linear functionals in $\bR^p$ with Huber loss function} \label{app:hub_reg}

To put into perspective the results obtained in Sections~\ref{section_local_ber_results}, we apply Theorem~\ref{thm:main} in the sub-Gaussian framework for linear regression in $\bR^p$. Let $F = \{ \inr{t, \cdot}, t \in  \bR^p  \}$, which satisfies assumption~\ref{assum:convex}. Let $(X_i,Y_i)_{i =1}^N$ be random variables defined by the following linear model:
\begin{equation} \label{linear_model_hub}
Y_i = \inr{X_i, t^*} + \epsilon_i \enspace,
\end{equation}
where $(X_i)_{i=1}^N$ are i.i.d Gaussian random vectors in $\bR^p$ with zero mean and covariance matrix $\Sigma$. The random variables $(\epsilon_i)_{i \in \cI}$ are assumed to be symmetric and independent to $(X_i)_{i=1}^N$. For the moment, nothing more is assumed for $(\epsilon_i)_{i \in \cI}$. 
It is clear that assumption~\ref{assum:distri} holds. The Empirical Risk Minimizer with the Huber loss function is defined as
\begin{equation} \label{erm_huber}
\hat t_N^{\gamma} =  \argmin_{t \in \bR^p}  \frac{1}{N} \sum_{i=1}^N\ell^{\gamma} (\inr{X_i,t},Y_i)
\end{equation}
where $\ell^{\gamma}(\cdot,\cdot)$ is the Huber loss function defined for any $\gamma >0$, $u,y\in\cY = \bR$, by
\[
\ell^{\gamma}(u,y) =  
\begin{cases}
\frac{1}{2}(y-u)^2&\text{ if }|u-y| \leq \gamma \\
\gamma|y-u|-\frac{\gamma^2}{2}&\text{ if }|u-y| > \gamma
\end{cases}\enspace,
\]
which satisfies assumption~\ref{assum:lip_conv} for $L = \gamma$. Let $t,v \in \bR^p$ such that $f(\cdot) = \inr{t,\cdot}$ and $g(\cdot) = \inr{v,\cdot}$. Since $\mu = \cN(0,\Sigma)$, we have
$\|f-g\|_{L_2}^2 = \bE \inr{t-v,X_1}^2 = (t-v)^T \Sigma (t-v)$ and $\lambda (f(X_1)- g(X_1))/ \|f-g\|_{L_2} = \big(\lambda/(t-v)^T \Sigma (t-v)\big) (t-v)^T X_1 \sim \cN(0, \lambda^2)$. If follows that $F-f^*$ is $1$-sub-Gaussian. \\

Let us turn to the computation of the complexity parameter $r^{SG}(A,\delta)$, for $A,\delta >0$. Well-known computations (see~\cite{talagrand2014upper}) give:
\begin{equation*}
w \big(F\cap  (f^*+ r\bB_2) \big) \leq  r \sqrt{\textnormal{Tr}(\Sigma)} \quad \mbox{and} \quad r_{\cI}^{SG}(A) =  cA  \gamma  \sqrt{\frac{\textnormal{Tr}(\Sigma)}{N} } \enspace,
\end{equation*}
for $c >0$ an absolute constant. \\

 To apply Theorem~\ref{thm:main}, it remains to study the local Bernstein assumption for the Huber loss function. We recall the following result from~\cite{ChiLecLer:2018}.
\begin{Proposition}[\cite{ChiLecLer:2018},Theorem 7] \label{prop:bernstein_huber_loss} Let $r>0$ and let $F_{Y|X=x}$ be the conditional cumulative function of $Y$ given $X=x$. Let us assume that the following holds.
	\begin{itemize}
		\item[a)] There exist  $\varepsilon,C' >0$ such that, for all $f$ in $F$, $\|f-f^*\|_{L_{2+\varepsilon}} \leq C' \|f-f^*\|_{L_2}$.
		\item[b)] Let $\varepsilon,C'$ be the constants defined in a). There exists $\alpha>0$ such that, for all $x\in \bR^p$ and all $z\in \mathbb{R}$ satisfying $ |z-f^*(x) | \leq (\sqrt{2} (C'))^{(2+\varepsilon)/\varepsilon} r$, $F_{Y|X=x}(z+\gamma) - F_{Y|X=x}(z- \gamma)\geqslant \alpha$.
	\end{itemize}
	Then, for all $f \in F \cap (f^* + r \bB_2) $,  $ (4/\alpha) P\cL^{\gamma}_f\geq \norm{f-f^*}_{L_2}^2$, where $P\cL^{\gamma}_f$ denotes the excess risk associated with the Huber loss with parameter $\gamma >0$. 
\end{Proposition}
Since $\mu = \cN(0,\Sigma)$, the point a) holds with $C' =3$. Moreover, from the model~\eqref{linear_model_hub}, the point b) can be rewritten as:  $\forall x \in \bR^p, \forall z \in \bR: |z - \inr{x,t^*} | \leq 18r$,
\begin{equation*}
	\bP  \bigg(  z-\gamma \leq \inr{x,t^*} + \epsilon \leq z +\gamma  \bigg)  = F_{\epsilon} (z+\gamma - \inr{x,t^*}) - F_{\epsilon} (z- \gamma - \inr{x,t^*})\geq \alpha 
\end{equation*} 
which is satisfied if
\begin{equation} \label{condition_bernstein}
F_{\epsilon} (\gamma -18r) - F_{\epsilon} (18r - \gamma)\geq \alpha 
\end{equation}
where $F_{\epsilon}$ denotes the cumulative distribution of $\epsilon$ distributed as $\epsilon_i$, for any $i \in \cI$. 
Condition~\eqref{condition_bernstein} simply implies that the noise puts enough mass around zero. \\
We are now in position to apply Theorem~\ref{thm:main} for Huber's $ M$-estimator in $\bR^p$.
\begin{Theorem} \label{thm:huber}
	Let $\cI \cup \cO$ denote a partition of $\{1, \cdots, N \}$ such that $|\cI|\geq |\cO|$. Let $(X_i,Y_i)_{i= 1}^N$ be random variables valued in $\bR^p \times \bR$ such that $(X_i)_{i=1}^N$ are i.i.d random variable with $X_1 \sim \cN(0,\Sigma)$ and for all $i \in \{ 1,\cdots, N \}$
	\begin{equation*}
		Y_i = \inr{X_i, t^*} + \epsilon_i \enspace.
	\end{equation*}
	Let
	\[
	r(\alpha, \delta) = c\frac{\gamma}{\alpha} \bigg( \sqrt \frac{\textnormal{Tr}(\Sigma) \vee \log(1/\delta)}{N} \vee \frac{| \cO|}{N} \bigg)  \enspace.
	\]
	Let $(\epsilon_i)_{i \in I}$ be i.i.d symmetric random variables independent to $(X_i)_{i \in \cI}$ such that there exists $\alpha>0$ such that
	\begin{equation} \label{cond_bernstein_thm}
	F_{\epsilon} \bigg(\gamma -  18r(\alpha, \delta)   \bigg) - F_{\epsilon} \bigg( 18r(\alpha, \delta)   - \gamma \bigg) \geq \alpha 
	\end{equation}
	where $	F_{\epsilon} $ denotes the cdf of $\epsilon$ distributed as $\epsilon_i$ for $i$ in $\cI$. With probability larger than $1-\delta$ the estimator $\hat{t}_N^{\gamma}$ defined in Equation~\eqref{erm_huber} satisfies
	\begin{align*}
	&	\| \Sigma^{1/2}(\hat t^{\gamma}_N - t^*)  \|_{2} \leq r(\alpha, \delta) \quad \mbox{and} \quad   P\cL_{\hat t^{\gamma}_N} \leq   c \alpha r^2(\alpha,\delta) 
	\end{align*}
\end{Theorem}
Theorem~\ref{thm:huber} holds under no assumption on $|\cO|$ except $|\cO|\leq |\cI|$. There are two situations
\begin{enumerate}
	\item The number of outliers $|\cO|$ is smaller than $\sqrt{\textnormal{Tr}(\Sigma) N}$. We obtain the rate of convergence $\gamma \sqrt{\textnormal{Tr}(\Sigma)/N}$. When $\bE [\epsilon_i^2] = \sigma^2$, $i \in \cI$ and $\gamma = \sigma$, it corresponds to the minimax-optimal rate of convergence. 
	\item The number of outliers $|\cO|$ exceeds $\sqrt{\textnormal{Tr}(\Sigma)N}$. In this case, the error rate and the excess risk are deteriorated and the dependence is linear with respect to the proportion of outliers. 
\end{enumerate}
Let $\varepsilon =  |\cO|/N$. From~\cite{chen2018robust}, this rate is minimax-optimal in the $\varepsilon$-Huber contamination model and hence also minimax-optimal when $|\cO|$ outliers contaminate only the labels (see Theorem~\ref{minimax_rate}). In Section~\ref{sec:simu}, we run simple simulations to illustrate the linear dependence between the error rate and the proportion of outliers. \\

Theorem~\ref{thm:huber} handles many different distributions for the noise as long as Equation~\eqref{cond_bernstein_thm} is satisfied. We illustrate the fact that the local Bernstein condition is very weak with the following example. Let $\epsilon \sim C(1)$ be a standard Cauchy distribution. For all $t \in \bR$, $F_{\epsilon}(t) = 1/2 + \arctan(t)/\pi$. From easy computations, Equation~\eqref{condition_bernstein} can be rewritten as 
\begin{equation} \label{cauchy_cond}
\arctan( \gamma - 18r) \geq \pi \alpha/2 \enspace.
\end{equation} 
For 
\[
r(\alpha,\delta) =  \frac{\gamma}{\alpha} \bigg( \sqrt \frac{\textnormal{Tr}(\Sigma) \vee \log(1/\delta)}{N} \vee \frac{| \cO|}{N} \bigg) \enspace,
\]
Equation~\eqref{cond_bernstein_thm} becomes
\begin{equation*}
	\arctan \bigg( \gamma \bigg[ 1 - \frac{c}{\alpha} \bigg( \sqrt \frac{\textnormal{Tr}(\Sigma) \vee \log(1/\delta)}{N} \vee \frac{| \cO|}{N}\bigg) \bigg]  \bigg) \geq  \pi \alpha /2 \enspace,
\end{equation*}
which is satified for $\alpha = 1/4$ and $\gamma = 2 \tan(\pi/8) $  if 
\begin{equation*}
c \bigg( \sqrt \frac{\textnormal{Tr}(\Sigma) \vee \log(1/\delta)}{N} \vee \frac{| \cO|}{N} \bigg) \leq 1  \enspace.
\end{equation*}

\begin{Proposition}
	In same framework as in Theorem~\ref{thm:huber}, when $\epsilon_i \sim C(1)$, for $i \in \cI$, the local Bernstein condition is verified for $\alpha =1/4$ and $\gamma = 2\tan(\pi/8)$ if
	\[
	\bigg( \sqrt \frac{\textnormal{Tr}(\Sigma) \vee \log(1/\delta)}{N} \vee \frac{| \cO|}{N} \bigg) \leq 1
	\]
\end{Proposition}

\begin{Remark}
The local Bernstein condition also holds many other noise distributions. In the case of Gaussian noise we can take $\alpha$ as a constant and $\gamma = \sigma$. In this case we recover the rate $\sigma \sqrt{ \textnormal{Tr}(\Sigma)/N}$.
\end{Remark}


\section{High dimensional setting} \label{sec:high_dim}

In Section~\ref{sec:no_reg}, we studied non-regularized procedures. If the class of predictors $F$ is too small there is no hope to approximate $Y$ with $f^*(X)$. It is thus necessary to consider large classes of functions leading to a large error rate unless some extra low-dimensional structure is expected on $f^*$. Adding a regularization term to the empirical loss is a wide-spread method to induce this low-dimensional structure. More formally, let $F \subset E \subset L_2$ and $\|\cdot\| \mapsto \bR^+$ be a norm defined on the linear space $E$. For any $\lambda > 0$, the regularized empirical risk minimizer (RERM) is defined as
\begin{equation} \label{def:RERM}
\hat{f}_N^{\lambda} = \argmin_{f \in F} \frac{1}{N} \sum_{i=1}^{N} \ell(f(X_i),Y_i) + \lambda \|f\|
\end{equation}
For example, the use of the $\ell_1$-norm promotes sparsity~\cite{tibshirani1996regression} for regression and classification problems in $\bR^p$, while the 1-Schatten norm promotes low rank solutions for matrix reconstruction. 
The main result of this section has the same flavor as the one in Section~\ref{sec:no_reg}. The error rate can be bounded by
\begin{equation*}
	r_N + AL \frac{|\cO|}{N} \enspace.
\end{equation*}
where $r_N$ denotes the (sparse or low-dimensional) error rate in a non contaminated setting, $L$ is the Lipschitz constant from Assumption~\ref{assum:lip_conv} and $A$ is a parameter coming from the local Bernstein condition. When  $| \cO| \leq r_N N/(AL)$, the RERM behaves as if there was no contamination.\\

\subsection{Complexity parameters and sparsity equation}

To analyze regularized procedures, we first need to redefine the complexity parameter.
\begin{Definition}\label{def:comp_reg}
	Let $\bB$ be the unit ball induced by the regularization norm $\|\cdot\|$. For any $A,\rho >0$, let $\tilde r_{\cI}^{SG}(A,\rho)$ and $\tilde r_{\cI}^{B}(A,\rho)$ be defined as 
	\begin{equation*}
		\tilde r_{\cI}^{SG}(A,\rho) = \inf \{ r>0 :  cALw \big(F\cap  (f^*+ r\bB_2 \cap \rho \bB) \big) \leq  \sqrt{|\cI|}r^2  \} \enspace,
	\end{equation*}
	and,
	\begin{equation*}
	\tilde r_{\cI}^{B}(A,\rho) = \inf \{ r>0 :  cAL\textnormal{Rad}_{\cI} \big(F\cap  (f^*+ r\bB_2 \cap \rho \bB) \big) \leq | \cI | r^2  \} \enspace,
	\end{equation*}
	where $c>0$ denotes an absolute constant and $L$ is the Lipschitz constant from assumption~\ref{assum:lip_conv}. For any $A,\delta,\rho >0$ let $\tilde r^{SG}(A,\rho,\delta)$ and $\tilde r^{B}(A,\rho,\delta)$ be such that 
	\begin{equation} \label{comp_par_regu_upper}
			\tilde r^{SG}(A,\rho,\delta) \geq \tilde{r}_{\cI}^{SG}(A,\rho) \vee AL \sqrt \frac{\log(1/\delta)}{N} \vee AL \frac{|\cO|}{N} \enspace,
	\end{equation}
	and,
	\begin{equation} \label{comp_par_regu_upper_bounded}
	\tilde r^B(A,\rho,\delta) \geq \tilde{r}_{\cI}^B(A,\rho) \vee AL \sqrt \frac{\log(1/\delta)}{N} \vee AL \frac{|\cO|}{N} \enspace.
	\end{equation}
\end{Definition}
The main difference between the complexity parameters from Definition~\ref{def:comp_reg}  and the ones from Definition~\ref{def:comp} is the localization $ \rho \bB $. Parameters in Definition~\ref{def:comp_reg} measure the local complexity of $F$ around $f^*$, where the localization is defined with respect to the metric induced by the regularization norm. \\

To deal with the regularization part, we use the tools from~\cite{LM_sparsity}. The idea is the following: the $\ell_1$ norm induces sparsity properties because it has large subdifferentials at sparse vectors. Therefore, to obtain “sparsity dependent bounds", i.e bounds depending on the unknown sparsity of the \textit{oracle} $f^*$, it seems natural to look at the size of the subdifferential of $\|\cdot\|$ in $f^*$. We recall that the subdifferential of $\|\cdot\|$ in $f$ is defined as
\begin{equation*} 
	(\partial \|.\|)_f = \{ z^{*} \in E^{*}  \enspace : \enspace \|f+h\| - \|f\| \geq z^{*}(h) \enspace \text{for every  } h \in E  \}\enspace,
\end{equation*}
where $E^{*}$ is the dual space of the normed space $(E,\|\cdot\|)$. It can be also written as
\begin{equation}\label{eq:sub_diff_norm}
(\partial \norm{\cdot})_f=\left\{
\begin{array}{cc}
\{z^*\in \bS^*:z^*(f) =\norm{f}\} & \mbox{ if } f\neq0\\
\bB^* & \mbox{ if } f=0
\end{array}
\right.
\end{equation}where $\bB^*$ and $\bS^*$ denote respectively the unit ball and the unit sphere with respect to the dual norm $\norm{\cdot}^*$ defined as $z ^*\in E^* \to\norm{z^*}^* = \sup_{\norm{f}\leq 1} z^*(f)$. When $f\neq0$, the subdifferential of $\norm{\cdot}$ in $f$ is the set of all vectors $z^*$ in the unit dual sphere $\bS^*$ which are norming for $f$. For any $\rho >0$, let
\begin{align*}
	\Gamma_{f^{*}}(\rho) = \bigcup_{f \in F \cap (f^* + (\rho/20) \bB) } (\partial \|\cdot\|)_f\enspace.
\end{align*}
Instead of looking at the subdifferential of $\|\cdot\|$ exactly in $f^*$ we consider the collection of subdifferentials for functions $f \in F$ “close enough" to the \textit{oracle} $f^*$. It enables to handle \textit{oracles} $f^*$ that are not exactly sparse but approximatively sparse. The main technical tool to analyze regularization procedures is the following sparsity equation~\cite{LM_sparsity}.
\begin{Definition}\label{def:SE} 
	For any $A,\rho,\delta >0$, let $\tilde r(A,\rho,\delta) \in \{ \tilde r^{SG}(A,\rho,\delta),\tilde r^B(A,\rho,\delta)   \}  $. Define
	\begin{gather} 
		\notag {H}_{A,\rho,\delta} =  F \cap \big(f^* + \rho \bB \cap \tilde r (A,\rho, \delta) \bB_2 \big)   \enspace,
	\end{gather}
	and
	\begin{equation} \label{sparisty:eq}
	\Delta(A,\rho, \delta) = \inf_{h \in H_{A,\rho, \delta} } \sup_{z^{*}  \in \Gamma_{f^{*}}(\rho) }   z^{*}(h-f^{*})\enspace.
	\end{equation}
	A real number $\rho>0$ satisfies the $A,\delta$-\textbf{sparsity equation} if $\Delta(A,\rho,\delta) \geq 4\rho /5$.
\end{Definition}
The constant $4/5$ in Definition~\ref{def:SE} could be replaced by any constant  in $(0,1)$. The sparsity equation is a very general and powerful tool allowing to derive “sparsity dependent bounds" when taking $\rho^*$ function of the unknown sparsity (see Section~\ref{app:huber_regu} for a more explicit example or~\cite{ChiLecLer:2019,LM_sparsity} for many other illustrations).
\begin{Remark} \label{remark_comp_bound}
	It is also possible to obtain “norm dependent bounds", i.e bounds depending on the norm of the \textit{oracle} $\|f^*\|$. By taking $\rho^* = 20 \|f^*\| $,  we get that  $ 0 \in  F \cap (f^* + (\rho^*/20) \bB)  $ and from Equation~\eqref{eq:sub_diff_norm} it follows that $\Gamma_{f^{*}}(20\|f^*\|) = \bB^*$ and for any $A,\delta >0$, $\Delta(A, \rho^* ,\delta) = \rho^*$. In other words, the sparsity equation is always satisfied for $\rho^* = 20 \|f^*\|$ (see  Section~\ref{app:rkhs} for an example)
\end{Remark}

\subsection{Local Bernstein conditions and main results}

In this section, we adapt the local Bernstein assumption to regularized framework. 

\begin{Assumption}\label{assum:fast_rates_reg} Let $\delta >0$ and $\tilde r(\cdot,\cdot,\delta) \in \{ \tilde r^{SG} (\cdot,\cdot,\delta),\tilde r^B(\cdot,\cdot,\delta)  \}$. Suppose there exist $A > 0$ and $\rho^*$ satisfying the $A,\delta$-sparsity equation from Definition~\ref{def:SE} such that for all $f \in F \cap \big( f^* + \rho \bB \cap \tilde r(A,\rho^*,\delta)  \bS_2 \big)$ we have $\|f-f^{*}\|_{L_2}^2\leq AP\mathcal{L}_f $.
\end{Assumption}
We are now in position to state the main theorem of this section.
\begin{Theorem} \label{thm:reg}
	Let $\cI  \cup \cO$ denote a partition of $\{1,\cdots, N\}$ such that $|\cO| \leq |\cI|$.  Grant Assumptions~\ref{assum:distri},~\ref{assum:convex},~\ref{assum:lip_conv}. Let $\delta >0$.
	\begin{enumerate}
			\item Let $\tilde r(\cdot,\cdot, \delta) = \tilde r^{SG}(\cdot,\cdot, \delta)$. \\
			Assume that the class $F-f^*$ is $1$-sub-Gaussian and that assumption~\ref{assum:fast_rates_reg} holds with $A, \rho^* >0$. Set:
			\begin{align*}
			\lambda = c \frac{(\tilde r^{SG} (A,\rho^*,\delta))^2}{A \rho^*} \enspace.
			\end{align*}	
			With probability larger than $1-\delta$, the estimator $\hat{f}_N^{\lambda}$ defined in Equation~\eqref{def:RERM} satisfies
			\begin{align*}
			& \|\hat{f}_N^{\lambda} - f^*\|_{{L_2}} \leq \tilde{r}^{SG}(A,\rho^*,\delta)  \quad , \quad 	\|\hat{f}_N^{\lambda} - f^*\| \leq \rho^* \quad \mbox{and}  \quad P\cL_{\hat{f}_N^{\lambda }} \leq   c  \frac{(\tilde r^{SG}(A,\rho^*,\delta))^2}{A} \enspace.
			\end{align*}
			\item  Let $\tilde r(\cdot,\cdot, \delta) = \tilde r^{B}(\cdot,\cdot, \delta)$.\\
			Let us assume that Assumption~\ref{assum:fast_rates_reg} holds with $A,\rho^* >0$ and that 
		\begin{equation} \label{cond_locally_bounded_reg}
		\forall f \in F \cap \big( f^* + \tilde r^B(A,\rho^*,\delta) \bS_2 \cap  \rho^* \bB   \big) \textnormal{ and } x \in \cX \quad |f(x)-f^*(x)| \leq 1	
		\end{equation}
		Set:
		\begin{align*}
		\lambda = c \frac{(\tilde r^{B} (A,\rho^*,\delta))^2}{A \rho^*} \enspace.
		\end{align*}	
		With probability larger than $1-\delta$, the estimator $\hat{f}_N^{\lambda}$ defined in Equation~\eqref{def:RERM} satisfies
		\begin{align*}
		& \|\hat{f}_N^{\lambda} - f^*\|_{{L_2}} \leq \tilde{r}^{B}(A,\rho^*,\delta)  \quad , \quad 	\|\hat{f}_N^{\lambda} - f^*\| \leq \rho^* \quad \mbox{and}  \quad P\cL_{\hat{f}_N^{\lambda }} \leq   c  \frac{(\tilde r^{B}(A,\rho^*,\delta))^2}{A} \enspace.
		\end{align*}
	\end{enumerate}
\end{Theorem}
When the equality holds in Equations~\eqref{comp_par_regu_upper} or~\eqref{comp_par_regu_upper_bounded} we have
\begin{equation*}
	\|\hat{f}_N^{\lambda} - f^*\|_{{L_2}} \leq   \tilde{r}_{\cI}(A,\rho) \vee AL \sqrt \frac{\log(1/\delta)}{N} \vee AL \frac{|\cO|}{N}   \enspace.
\end{equation*}	
with probability larger than $1-\delta$.
The error rate can be decomposed as the error rate in the non-contaminated setting and the proportion of outliers $AL|\cO|/N$. \\
Equation~\eqref{cond_locally_bounded_reg} means that the class $F-f^*$ is locally bounded. As we will see in Section~\ref{app:rkhs}, requiring the local boundedness instead of the global one has important consequences.\\
Theorem~\ref{thm:reg} is a “meta" theorem in the sense that it can used for many practical problems. We use Theorem~\ref{thm:reg} for $\ell_1$-penalized Huber's M-estimator in Section~\ref{app:huber_regu}. It is also possible to use  Theorem~\ref{thm:reg} for many other convex and Lipschitz loss functions and regularization norms as it is done in~\cite{ChiLecLer:2019}. It can also be used for matrix reconstruction problems by penalizing with the 1-Schatten norm~\cite{LM_sparsity}. \\

Theorem~\ref{thm:reg} may seem complicated at a first glance because the parameter $A$ appears in the definition of the complexity parameters, in the sparsity equation and in Assumption~\ref{assum:fast_rates_reg}. However, there is a simple general routine that we may use to apply Theorem~\ref{thm:reg}.

\paragraph{General routine to apply Theorem~\ref{thm:reg} when the class $F-f^*$ is sub-Gaussian}
\begin{enumerate}
	\item Verify that the class $F-f^*$ is sub-Gaussian and take $\tilde r(\cdot,\cdot,\cdot) = \tilde r^{SG}(\cdot,\cdot,\cdot)$.
	\item Verify assumptions~\ref{assum:distri},~\ref{assum:convex} and~\ref{assum:lip_conv}.
	\item Compute the localized Gaussian mean width $w \big(F  \cap (f^* + r\bB_2 \cap \rho \bB ) \big)$ for any $r,\rho >0$. Deduce the value of $\tilde{r}_{\cI}^{SG}(A,\rho)$ for any $A,\rho > 0$.
	\item From the computation of $\tilde{r}_{\cI}^{SG}(A,\rho)$ deduce the closed form of $\tilde r^{SG}(A,\rho,\delta)$. 
	\item For fixed constants $A, \delta > 0$, find $\rho^* >0$ satisfying the $A,\delta$- sparsity equation.
	\item From the value of $\rho^*$, compute $\tilde{r}^{SG}(A,\rho^*,\delta)$ for any $A,\delta >0$.
	\item Find a constant $ A >0$ verifying Assumption~\ref{assum:fast_rates_reg}. 
\end{enumerate}

\paragraph{General routine to apply Theorem~\ref{thm:reg} when the class $F-f^*$ is locally bounded}
\begin{enumerate}
	\item Take $\tilde r(\cdot,\cdot,\cdot) = \tilde r^{B}(\cdot,\cdot,\cdot)$.
	\item Verify assumptions~\ref{assum:distri},~\ref{assum:convex} and~\ref{assum:lip_conv}.
	\item Compute the localized Rademacher complexity $\textnormal{Rad}_{\cI} \big(F  \cap (f^* + r\bB_2 \cap \rho \bB ) \big)$ for any $r,\rho >0$. Deduce the value of $\tilde{r}_{\cI}^{B}(A,\rho)$ for any $A,\rho > 0$.
	\item From the computation of $\tilde{r}_{\cI}^{B}(A,\rho)$ deduce the closed form of $\tilde r^{G}(A,\rho,\delta)$. 
	\item For fixed constants $A, \delta > 0$, find $\rho^* >0$ satisfying the $A,\delta$- sparsity equation.
	\item From the value of $\rho^*$, compute $\tilde{r}^B(A,\rho^*,\delta)$ for any $A,\delta >0$.
	\item Find a constant $ A >0$ verifying Assumption~\ref{assum:fast_rates_reg}. 
	\item Verify that the class $F-f^*$ is locally bounded with $\tilde{r}^B(A,\rho^*,\delta)$ computed previously. 
\end{enumerate}
We will apply these two general routines for practical examples in Section~\ref{app:huber_regu} and~\ref{app:rkhs}.

\subsection{Application to $\ell_1$-penalized Huber's M-estimator with Gaussian design} \label{app:huber_regu}

Let $F = \{ \inr{t,\cdot}, t \in \bR^p  \}$ denote the class of linear functionals in $\bR^p$. Let $(X_i,Y_i)_{i =1}^N$ be random variables defined by, $Y_i = \inr{X_i, t^*} + \epsilon_i $, where $(X_i)_{i=1}^N$ are i.i.d centered standard Gaussian vectors. The random variables $(\epsilon_i)_{ i \in \cI}$ are symmetric independent to $(X_i)_{i \in \cI}$. The \textit{oracle} $t^*$ is assumed to be $s$-sparse, $\|t^*\|_0 := \sum_{i=1}^p   \mathbb I \{ t^*_i \neq 0  \} \leq s$. \\
The $\ell_1$-penalized Huber's M-estimator is defined as 
\begin{equation} \label{rerm_huber}
\hat t_N^{\gamma,\lambda} =  \argmin_{t \in \bR^p}  \frac{1}{N} \sum_{i=1}^N\ell^{\gamma}(\inr{X_i,t},Y_i) + \lambda  \|t\|_1
\end{equation}
where $\ell^{\gamma}(\cdot,\cdot)$ is the Huber loss function. We use the routine of Theorem~\ref{thm:reg} when $F-f^*$ is sub-Gaussian:\\
\textbf{Step 1: } As in Section~\ref{app:hub_reg}, the class $F-f^*$ is $1$-sub-Gaussian.\\
\textbf{Step 2 } It is clear that Assumptions~\ref{assum:distri},~\ref{assum:convex},~\ref{assum:lip_conv} with $L = \gamma$ are satisfied. \\
\textbf{Step 3 and 4: }Let us turn to the computation of the local Gaussian-mean width. Since $X \sim \cN(0,I_p)$, for every $t\in \bR^p$, we have $w \big(F  \cap (f^* + r\bB_{2} \cap \rho \bB ) \big) = w(r \bB_2^p \cap \rho \bB_1^p )$ for every $r,\rho >0$, where $\bB_q^p$ denotes the $\ell_q$ ball in $\bR^p$ for $q >0$. Well-known computations give (see~\cite{vershynin2018high} for example)
\begin{equation*}
	w (\rho B_1^p \cap r B_2^p)  \leq \rho w ( B_1^p)\leq c \rho \sqrt{\log(p)} \enspace, 
\end{equation*} 
and consequently,
\begin{equation*} 
	\big(\tilde r_{\cI}^{SG}(A, \rho) \big)^2 = c A \gamma  \rho \sqrt{\frac{\log(p)}{N}} \enspace,
\end{equation*}
and let $\tilde r^{SG}(A,\rho,\delta)$ be such that
\[
 \tilde r^{SG}(A,\rho,\delta)  \geq c \bigg(  \sqrt{A\gamma \rho }  \bigg(\frac{\log(p)}{N} \bigg)^{1/4} \vee A\gamma \sqrt \frac{\log(1/\delta)}{N} \vee A \gamma  \frac{| \cO |}{N}\bigg)
\]
\textbf{Step 5 and 6:} To verify the $A,\delta$-sparsity equation from Definition~\ref{def:SE} for the $\ell_1$ norm we use the following result from \cite{LM_sparsity}.
\begin{Lemma}\cite[Lemma 4.2]{LM_sparsity} \label{lemma_lasso}.
	Let $\bB_1^p$ denote the unit ball induced by $\|\cdot\|_1$. Let us assume that the design $X$ is isotropic. If the \textit{oracle} $t^{*}$ is $s$-sparse and $100s \leq \big(\rho / \big(   \tilde r^{SG}(A,\rho,\delta)  \big)^2$ then $\Delta(A,\rho,\delta) \geq (4/5)\rho$.
\end{Lemma}
Lemma~\ref{lemma_lasso} implies that the $A,\delta$-sparsity equation is satisfied with $\rho^* >0$ if the sparsity $s$ is smaller than $\big(\rho^* / \big(   \tilde r^{SG}(A,\rho^*,\delta)  \big)^2$. From easy computations, it follows 
\[
 \rho^* = A  \gamma \bigg( s \sqrt \frac{\log(p)}{N} \vee \sqrt \frac{s \log(1/\delta)}{N} \vee  \sqrt s \frac{|\cO|}{N}\bigg)\enspace,
\]
and
\[
\tilde r^{SG}(A,\rho^*,\delta) = A \gamma \bigg(  \sqrt{ \frac{s\log(p) \vee \log(1/\delta)}{N}} \vee  \frac{|\cO|}{N}\bigg) \enspace.
\]
\textbf{Step 7 :} We use Proposition~\ref{prop:bernstein_huber_loss} to show that the local Bernstein condition holds for $f \in F \cap \big( f^* + \tilde r^{SG}(A,\rho^*,\delta) \bS_2 \cap \rho^* \bB  \big)$. Since $X \sim \cN(0,I_p)$, the point a) in Proposition~\ref{prop:bernstein_huber_loss} is verified. Moreover, the point b) holds and the local Bernstein condition is verified with $A=4/\alpha$  if there exists $\alpha > 0$ satisfying
\begin{equation} \label{cond_noise_hub}
F_{\epsilon} \bigg( \gamma -  c\tilde r^{SG}(4/\alpha, \rho^*,\delta)  \bigg) - F_{\epsilon} \bigg( c  \tilde r^{SG} (4/\alpha, \rho^*,\delta)- \gamma \bigg) \geq \alpha \enspace, 
\end{equation}
where $F_{\epsilon}$ denotes the cdf of $\epsilon$ distributed as $\epsilon_i$, for $i \in \cI$.\\

We are now in position to state the main result for the $\ell_1$-penalized Huber estimator. 
\begin{Theorem} \label{thm:huber_pen}
	Let $\cI \cup \cO$ denote a partition of $\{1, \cdots, N \}$ such that $|\cI|\geq |\cO|$ and $(X_i,Y_i)_{i= 1}^N$ be random variables valued in $\bR^p \times \bR$ such that $(X_i)_{i=1}^N$ are i.i.d random variable with $X_1 \sim \cN(0,I_p)$ and for all $i \in \{ 1,\cdots, N \}$
	\begin{equation*}
		Y_i = \inr{X_i, t^*} + \epsilon_i \enspace,
	\end{equation*}
	where $t^*$ is $s$-sparse. For any $\delta,\alpha > 0$, let
	\[
	\tilde r^{SG}(\alpha,\delta)  =  c \frac{\gamma}{\alpha} \bigg(  \sqrt{ \frac{s\log(p) \vee \log(1/\delta)}{N}} \vee  \frac{|\cO|}{N}\bigg) 
	\]
	Let $(\epsilon_i)_{i \in I}$ are i.i.symmetric random variables independent to $(X_i)_{i \in \cI}$ such that there exists $\alpha>0$ such that
	\begin{equation} \label{cond_hub_noise}
	F_{\epsilon} \big(\gamma -   \tilde r^{SG}(\alpha,\delta)    \big) - F_{\epsilon} \big( 	\tilde r^{SG}(\alpha,\delta)  - \gamma \big)\geq \alpha 
	\end{equation}
	where $	F_{\epsilon} $ denotes the cdf of $\epsilon$, where $\epsilon$ is distributed as $\epsilon_i$, for $i$ in $\cI$. Set
	\begin{equation*}
		\lambda = c \gamma \bigg(  \sqrt{\frac{\log(p)}{N}} \vee \sqrt \frac{\log(1/\delta)}{s N} \vee\frac{|\cO| }{\sqrt sN}  \bigg) \enspace.
	\end{equation*}
	Then with probability larger than $1-\delta$, the estimator $\hat{t}_N^{\gamma,\lambda}$ defined in Equation~\eqref{rerm_huber} satisfies
	\begin{align*}
		&		\|\hat{t}_N^{\gamma,\lambda}- t^*\|_{2} \leq \tilde r^{SG}(\alpha,\delta) , \quad   P\cL_{\hat{t}_N^{\gamma,\lambda}} \leq c\alpha (\tilde r^{SG}(\alpha,\delta))^2 \\
		& \mbox{and } \quad \|\hat{t}_N^{\gamma,\lambda} - t^*\|_1 \leq c \frac{\gamma}{\alpha} \bigg( s \sqrt \frac{\log(p)}{N} \vee \sqrt \frac{s \log(1/\delta)}{N} \vee  \sqrt s \frac{|\cO|}{N}\bigg) 
	\end{align*}
\end{Theorem}
There are two situations:
\begin{enumerate}
	\item When the number of outliers $|\cO|$ is smaller than $\sqrt{s \log(p)N}$, the regularization parameter $\lambda$ does not depend on the unknown sparsity and we obtain the (nearly) minimax-optimal rate in sparse linear regression in $\bR^p$~\cite{bellec2016slope,LM_sparsity,dalalyan2017prediction}. Using more involved computations and taking a regularization parameter $\lambda$ depending on the unknown sparsity, we could obtain the exact minimax rate of convergence $s\log(p/s)/N$. 
	\item  When the number of outliers exceeds $\sqrt{s\log(p) N}$ the value of $\lambda$ depends on the unknown quantities $|\cO|$ and $s$. The error rate is deteriorated and becomes linear with respect to the proportion of outliers $|\cO| / N$. Using Theorem~\ref{minimax_rate} and~\cite{chen2016general}, this error rate is minimax optimal (up to a logarithmic term) when $|\cO|$ malicious outliers contaminate only the labels.
\end{enumerate} 
As in Section~\ref{app:hub_reg}, we can assume that the noise follows a standard Cauchy distribution. In this case we can take $\alpha =1/4$ and $\gamma = 2 \tan(\pi/8)$ and Equation~\eqref{cond_hub_noise} holds if
\begin{equation} \label{eq:end}
\sqrt{ \frac{s\log(p) \vee \log(1/\delta)}{N}} \vee  \frac{|\cO|}{N} \leq c \enspace.
\end{equation}
When $\epsilon_i  \sim \cN(0,\sigma^2)$, the local Bernstein condition is verified with $\alpha = 1/4$ and $\gamma = c\sigma$ if Equation~\eqref{eq:end} holds. 
In Section~\ref{sec:simu}, we run simple simulations to illustrate the linear dependence between the error rate and the proportion of outliers. 
\begin{Remark}
	In Theorem~\ref{thm:huber_pen} we assumed that $\mu = \cN(0,I_p)$ to apply Lemma~\ref{lemma_lasso} and compute the local Gaussian-mean width.  It is possible to generalize the result to Gaussian random vectors with covariance matrices $\Sigma$ verifying $RE(s,9)$ \cite{van2009conditions}, where $s$ is the sparsity of $t^*$. Recall that a matrix $\Sigma$ is said to satisfy the restricted eigenvalue condition RE$(s, c_0)$ with some constant $\kappa > 0$, if $\|\Sigma^{1/2} v \|_2  \geq \kappa \|v_J\|_2$ for any vector $v$ in $\bR^p$ and any set $J \subset \{1,\cdots,p \}$ such that $|J| \leq s$ and $\|v_{J^c}\|_1 \leq c_0 \|v_J\|_1$. When $\Sigma$ satisfies the RE$(s,9)$ condition with $\kappa >0$ we get the same conclusion as Theorem~\ref{thm:huber_pen} modulo an extra term $1/\kappa$ in front of $\tilde r_{\cI}(A,\rho^*,\delta)$ (see Section~\ref{RE_huber} for a precise result). 
\end{Remark}	


\subsection{Application to RKHS with the huber loss function} \label{app:rkhs}

We present another example of application of our main results. In particular, we use the routine associated with Theorem~\ref{thm:reg} in the locally bounded case, for the problem of learning in a Reproducible Kernel Hilbert Space (RKHS) $\cH_K$~\cite{steinwart2008support} associated to a bounded positive definite kernel $K$. We are given $N$ pairs $(X_i,Y_i)_{i=1}^N$ of random variables where the $X_i$'s take their values in some measurable space $\cX$ and $Y_i \in \bR$. We introduce a kernel $K: \cX \times \cX \mapsto \bR$ measuring a similarity between elements of $\cX$ i.e $K(x_1,x_2)$ is small if $x_1,x_2 \in \cX$ are “similar".  The kernel $K(\cdot,\cdot)$ is assumed to be bounded (for all $x \in \cX: |K(x,x)| \leq 1$). The main idea of kernel methods is to transport the design data $X_i$'s from the set $\cX$ to a certain Hilbert space via the application $x \mapsto K(x,\cdot) := K_x(\cdot)$ and construct a statistical procedure in this "transported" and structured space. The kernel $K$ is used to generate a Hilbert space known as Reproducing Kernel Hilbert Space (RKHS). Recall that if $K$ is a positive definite function i.e for all $n \in \bN^* $, $x_1,\cdots,x_n \in \cX$ and $c_1,\cdots,c_n \in \bR$, $\sum_{i=1}^n \sum_{j=1}^n c_i c_j K(x_i,x_j) \geq 0$. By Mercer's theorem there exists an orthonormal basis $(\phi_i)_{i=1}^{\infty}$ of $L_2$ such that $\mu \times \mu$ almost surely, $K(x,y) = \sum_{i=1}^{\infty} \lambda_i \phi_i(x) \phi_i(y)$, where $(\lambda)_{i=1}^{\infty}$ is the sequence of eigenvalues (arranged in a non-increasing order) of $T_K$ and $\phi_i$ is the eigenvector corresponding to $\lambda_i$ where 
\begin{align} \label{def_tk}
	\nonumber T_K: \; & L_2 \to L_2 \\
	& (T_Kf)(x) = \int K(x,y) f(y) d\mu (y) 
\end{align}
The Reproducing Kernel Hilbert Space $\cH_K$ is the set of all functions of the form $\sum_{i=1}^{\infty} a_i K(x_i,\cdot)$ where $x_i \in \cX$ and $a_i \in \bR$ converging in $L_2$ endowed with the inner product
\begin{equation*}
	\inr{\sum_{i=1}^{\infty} a_i K(x_i,\cdot),\sum_{i=1}^{\infty} b_i K(y_i,\cdot)} = \sum_{i,j=1}^{\infty} a_i b_j K(x_i,y_i)
\end{equation*}
An alternative way to define a RKHS is via the feature map $\Phi: \cX \mapsto \ell_2$ such that $\Phi(x) = \big(\sqrt{\lambda_i}\phi_i(x)\big)_{i=1}^{\infty} $. Since $(\Phi_k)_{k=1}^{\infty}$ is an orthogonal basis of $\cH_K$, it is easy to see that the unit ball of $\cH_K$ can be expressed as 
\begin{equation}
\bB_{\cH_K} = \{ f_{\beta}(\cdot) = \inr{\beta,\Phi(\cdot)}_{\ell_2}, \; \| \beta \|_{2}  \leq 1  \} 
\end{equation}
where $\inr{\cdot,\cdot}_{\ell_2}$ is the standard inner product in the  Hilbert space $\ell_2$. 
In other words, the feature map $\Phi$ can the used to define an isometry between the two Hilbert spaces $\cH_K$ and $\ell_2$. \\

The RKHS $\cH_K$ is  a convex class of functions from $\cX$ to $\bR$ that can be used as a learning class $F$. Let us assume that 
$Y_i = f^*(X_i)+ \epsilon_i $ where $(X_i)_{i=1}^N$ are i.i.d random variables taking values in $\cX$. The random variables $(\epsilon_i)_{ i \in \cI}$ are symmetric i.i.d random variables independent to $(X_i)_{i \in \cI}$ and $f^*$ is assumed to belong to $\cH_K$. It follows that the \textit{oracle} $f^*$ is also defined as 
\begin{equation*}
	f^* \in \argmin_{f \in \cH_K} \bE [  \ell^{\gamma}(f(X),Y) ]
\end{equation*}
where  $\ell^{\gamma}$ is the Huber loss function. For the sake of simplicity, we assume that $\|f^*\|_{\cH_K} \leq 1$. Without this assumption, it is possible to obtain error rates depending on $\|f^*\|_{\cH_K}$. However, we do not pursue this analysis here to simplify the presentation. Let $f$ be in $\cH_K$. By the reproducing property and the Cauchy-Schwarz inequality we have for all $x,y$ in $\cX$
\begin{align} \label{bounded:RKHS}
	|f(x)-f(y)| = | \inr{f, K_x -K_y} | \leq \|f\|_{\cH_K}  \| K_x - K_y \|_{\cH_K}  
\end{align}
From Equation~\eqref{bounded:RKHS}, it is clear that the norm of a function in the RKHS controls how fast the function varies over $\cX$ with respect to the geometry defined by the kernel (Lipschitz with constant $\|f \|_{\cH_K})$. As a consequence the norm of regularization $\| \cdot \|_{\cH_K}$ is related with its degree of smoothness w.r.t. the metric defined by the kernel on $\cX$ and assuming that $\|f^*\|_{\cH_K} \leq 1$ is equivalent to assume that the \textit{oracle} $f^*$ is smooth enough. The estimators $\hat{f}_{N}^{\gamma,\lambda}$ we study in this section is defined as
\begin{equation} \label{def_svm}
\hat{f}_{N}^{\gamma,\lambda}  = \argmin_{f \in \cH_K} \frac{1}{N} \sum_{i=1}^{N}  \ell^{\gamma}(f(X_i),Y_i) + \lambda \|f\|_{\cH_K}
\end{equation}
We obtain error rates depending on spectrum $(\lambda_i)_{i=1}^{\infty}$ of the integral operator $T_K$ assumed to satisfy the following assumption. 
\begin{Assumption} \label{assum:spectrum}
	The eigenvalues $(\lambda_i)_{i=1}^{\infty}$ of the integral operator $T_K$ satisfy $\lambda_n   \leq c n^{-1/p}$ for some $0 < p < 1$ and $c>0$ an absolute constant.
\end{Assumption}
In Assumption~\ref{assum:spectrum}, the value of $p$ is related with the smoothness of the space $\cH_K$. Different kinds of spectra could be analyzed. It would only change the computation of the complexity fixed-points. For the sake of simplicity we only focus on this example as it has been also studied in~\cite{caponnetto2007optimal,mendelson2010regularization} to obtain fast rates of convergence. \\

Let us use the routine to apply Theorem~\ref{thm:reg} in the locally bounded setting. Indeed, we will see than the localization with respect to the norm induced by the kernel allows to obtain a locally bounded class of functions. \\
\textbf{Step 1:} For any $A,\rho,\delta >0,$ let $\tilde r(A,\rho,\delta) = \tilde r^B(A,\rho,\delta)$. \\
\textbf{Step 2:} Since every Reproducible Kernel Hilbert space is convex, it is clear that assumptions~\ref{assum:distri},~\ref{assum:convex} and~\ref{assum:lip_conv} with $L = \gamma$ are satisfied. \\
\textbf{Step 3:} From Theorem 2.1 in~\cite{mendelson2003performance},  for all $\rho,r >0$
\begin{equation*}
	\textnormal{Rad}_{\cI} \big(\cH_K  \cap (f^*+ r\bB_2 \cap \rho \bB_{\cH_K} \big)  \leq c \sqrt{|\cI|} \bigg( \sum_{k=1}^{\infty} \big(  \rho^2 \lambda_k \wedge r^2 \big) \bigg)^{1/2} \enspace.
\end{equation*}
Under assumption~\ref{assum:spectrum}, straightforward computations give,
\begin{equation*}
	\bigg( \sum_{k=1}^{\infty} \big(  \rho^2 \lambda_k \wedge r^2 \big) \bigg)^{1/2} \leq c \frac{\rho^p}{r^{p-1}},
\end{equation*}
and thus for any $A,\rho >0$
\begin{equation*}
	\tilde r^B_{\cI}(A,\rho) = c (A\gamma)^{1/(p+1)}  \frac{\rho^{p/(p+1)}}{N^{1/(2(p+1))}}
\end{equation*}
\textbf{Step 4:} It follows that
\[
\tilde r^B(A,\rho,\delta) = c \bigg( (A\gamma)^{1/(p+1)}  \frac{\rho^{p/(p+1)}}{N^{1/(2(p+1))}} \vee A\gamma \sqrt \frac{\log(1/\delta)}{N} \vee A \gamma \frac{|\cO|}{N}   \bigg) 
\]
\textbf{Step 5: } From Remark~\ref{remark_comp_bound}, $\rho^* = 20\|f^*\|_{\cH_K} \leq 20 $ satisfies the $A,\delta$-sparsity equation for any $A,\delta >0$. \\
\textbf{Step 6:} From step 5, we easily get
\begin{equation*}
	\tilde r^b(A,\delta) = c \bigg(   \frac{(A\gamma)^{1/(p+1)}}{N^{1/(2(p+1))}} \vee A\gamma \sqrt \frac{\log(1/\delta)}{N} \vee A \gamma \frac{|\cO|}{N}   \bigg) 
\end{equation*}
\textbf{Step 7 } This step consists in verifying that Assumption~\ref{assum:fast_rates_reg} holds. To do so, we use a localized version of Theorem~\ref{prop:bernstein_huber_loss}.

\begin{Proposition}\label{prop:bernstein_huber_loss_loc}
 Let $r,\rho>0$ and let $F_{Y|X=x}$ be the conditional cumulative function of $Y$ given $X=x$. Let us assume that:
	\begin{itemize}
		\item[a)] There exist  $\varepsilon,C' >0$ such that, for all $f$ in $F \cap \big( f^* + \rho \bB_{\cH_K}  \cap r \bS_2 \big) $, we have $\|f-f^*\|_{L_{2+\varepsilon}} \leq C' \|f-f^*\|_{L_2}$.
		\item[b)] Let $\varepsilon,C'$ be the constants defined in a). There exists $\alpha>0$ such that, for all $x\in \bR^p$ and all $z\in \mathbb{R}$ satisfying $ |z-f^*(x) | \leq (\sqrt{2} (C'))^{(2+\varepsilon)/\varepsilon} r$, $F_{Y|X=x}(z+\gamma) - F_{Y|X=x}(z- \gamma)\geqslant \alpha$.
	\end{itemize}
Then, for all $f \in F \cap \big( f^* + \rho \bB_{\cH_K}  \cap r \bS_2  \big)$, $ (4/\alpha) P\cL_f^{\gamma} \geq \norm{f-f^*}_{L_2}^2$, where $P\cL_f^{\gamma}$ denotes the excess risk associated with the Huber loss function with parameter $\gamma >0$.
\end{Proposition}
The only difference with Proposition~\ref{prop:bernstein_huber_loss}, is that the point a)  and the conclusion hold for functions $f$ in $F \cap \big( f^* + \rho \bB_{\cH_K}  \cap r \bS_2 \big) $. Proposition~\ref{prop:bernstein_huber_loss_loc} is a refinement of Proposition~\ref{prop:bernstein_huber_loss} where a localization with respect to the regularization norm is added. \\
Let $f$ in $\cH_K$ such that $\|f-f^*\|_{\cH_K} \leq \rho$ and $\|f-f^*\|_{L_2} = r$. Since $|f(x) - g(x)| = |\inr{f-g,K_x}|$ for any $f,g \in \cH_K$, $x \in \cX$ we get
\begin{equation*}
	\|f-f^{*}\|_{L_{2+\varepsilon}}^{2+\varepsilon} = \int (f(x)-f^*(x))^{2+\varepsilon} dP_X(x) \leq  (\rho  )^{\varepsilon} \|f-f^*\|_{L_2}^{2}
\end{equation*}
Since $\|f-f^*\|_{L_2} = r$, it follows that
\begin{equation*}
	\|f-f^{*}\|_{L_{2+\varepsilon}} \leq  \bigg( \frac{\rho}{r} \bigg)^{\varepsilon/(2+\varepsilon)}  \|f-f^*\|_{L_2}. 
\end{equation*} 
Therefore, the point a) holds with $C' =  (\rho/r)^{\varepsilon/(2+\varepsilon)}$. Let us turn to the point b). From the fact that $C'= (\rho /r)^{\varepsilon/(2+\varepsilon)}$, we have $(\sqrt{2}C')^{(2+\varepsilon)/\varepsilon} r = 2^{(2+\varepsilon)/2\varepsilon} \rho$ and the point b) can be rewritten as, there exists $\alpha >0$ such that
\begin{equation} \label{cond_bern_rkhs}
F_{\epsilon} ( \gamma - c \rho ) - F_{\epsilon} ( c \rho  - \gamma ) \geq \alpha 
\end{equation}
where $F_{\epsilon}$ denotes the cdf of $\epsilon$ distributed as $\epsilon_i$ for $i \in \cI$. Equation~\eqref{cond_bern_rkhs}, simply means that the noise $\epsilon$ puts enough mass around $0$. In this setting we have $\rho= \rho^* = c$ and Equation~\eqref{cond_bern_rkhs} becomes,
\begin{equation*} 
	F_{\epsilon} ( \gamma - c) - F_{\epsilon} ( c  - \gamma ) \geq \alpha 
\end{equation*}
\textbf{Step 8: } Let us turn to the local boundedness assumption. Since $|f(x) - f^*(x)| = |\inr{f-f^*,K_x}|$ for any $f \in \cH_K$, $x \in \cX$, if $\|f-f^*\|_{\cH_K} \leq \rho^*$ we get $|f(x) - f^*(x)| \leq  \rho^* = c$ and the local boundedness assumption is well-satisfied. \\

The fact that the boundedness assumption is only required to hold locally is essential. It is obvious that it does not hold here over the whole class $F = \cH_{K}$. We are now in position to state our main theorem for regularized learning in RKHS with the Huber loss function. 
\begin{Theorem} \label{thm:huber_pen_rkhs}
	Let $\cH_K$ be a reproducible kernel Hilbert space associated with kernel $K$, where $|K(x,x)| \leq 1$, for any $x \in \cX$. Let $\cI \cup \cO$ denote a partition of $\{1, \cdots, N \}$ such that $|\cI|\geq |\cO|$ and $(X_i,Y_i)_{i= 1}^N$ be random variables valued in $\cX \times \bR$ such that $(X_i)_{i=1}^N$ are i.i.d random variable and for all $i \in \{ 1,\cdots, N \}$
	\begin{equation*}
		Y_i = f^*(X_i) + \epsilon_i \enspace,
	\end{equation*}
	where $f^*$ belongs to $\cH_K$ and $\|f^*\|_{\cH_K} \leq  1$.  Assume that $(\epsilon_i)_{i \in I}$ are i.i.d symmetric random variables independent to $(X_i)_{i \in \cI}$ such that there exists $\alpha>0$ such that
	\begin{equation} \label{cond_hub_noise_rkhs}
	F_{\epsilon} \big(\gamma -  c   \big) - F_{\epsilon} \big(   c   - \gamma \big)\geq \alpha 
	\end{equation}
	where $	F_{\epsilon} $ denotes the cdf of $\epsilon$ where $\epsilon$ is distributed as $\epsilon_i$, for $i$ in $\cI$. Grant Assumption~\ref{assum:spectrum} and let
	\[
		\tilde r(\alpha,\gamma) = c \bigg(   \frac{(\gamma/\alpha)^{1/(p+1)}}{N^{1/(2(p+1))}} \vee \frac{\gamma}{\alpha}  \sqrt \frac{\log(1/\delta)}{N} \vee  \frac{\gamma}{\alpha} \frac{|\cO|}{N}   \bigg) 
	\]
	Set $\lambda = c \alpha\tilde r(\alpha, \gamma)$. Then with probability larger than $1-\delta$, the estimator $\hat{f}_N^{\gamma,\lambda}$ defined in Equation~\eqref{def_svm} satisfies
	\begin{align*}
		&		\| \hat{f}_N^{\gamma,\lambda}- f^* \|_{2}^2  \leq \tilde r(\alpha, \gamma) \quad \textnormal{and} \quad  P\cL_{\hat{f}_N^{\gamma,\lambda}} \leq   c  \alpha \tilde r(\alpha, \gamma)
	\end{align*}
\end{Theorem}
Theorem~\ref{thm:huber_pen_rkhs} holds with no assumption on the design $X$. 

\begin{enumerate}
	\item When
\[
|\cO| \leq (\alpha/\gamma)^{p/(p+1)} N^{(2p+1)/(2p+2)} \enspace,
\]
we recover the same rates as \cite{smale2007learning,mendelson2010regularization} even when the target $Y$ is heavy-tailed. In \cite{smale2007learning,mendelson2010regularization} the authors assume that $Y$ is bounded while in \cite{caponnetto2007optimal} the noise is assumed to be light-tailed. We generalize their results to heavy-tailed noise. 
\item When 
\[
|\cO| \geq (\alpha/\gamma)^{p/(p+1)} N^{(2p+1)/(2p+2)} \enspace,
\]
 the error rate is deteriorated and becomes linear with respect to the proportion of outliers.\\ 
\end{enumerate}
When the noise is Cauchy distributed, we can take $\gamma$ a large enough absolute constant to verify Equation~\eqref{cond_hub_noise_rkhs}. We obtain an error rate of order $N^{-1/(p+1)}$. Depending on the value of $p$ we have obtained fast rates of convergence for regularized Kernel methods. The faster the spectrum of $T_K$ decreases the faster the rates of convergence. 

\section{Conclusion and perspectives}

We have presented general analyses to study ERM and RERM when $|\cO|$ outliers contaminate the labels when 1) the class $F-f^*$ is sub-Gaussian or 2) when the class $F-f^*$ is locally bounded. We use these “meta theorems" to study Huber's M-estimator with no regularization or penalized with the $\ell_1$ norm. Under a very weak assumption on the noise (note that it can even not be integrable), we have obtained minimax-optimal rate of convergence for these two examples when $|\cO|$ malicious outliers corrupt the labels. We also have obtained fast rates for regularized learning problems in RKHS when the target $Y$ is unbounded and heavy-tailed. \\
For the sake of simplicity, we have only presented two examples of applications. Many procedures can be analyzed as it has be done in~\cite{ChiLecLer:2019} such as Group Lasso, Fused Lasso, SLOPE etc. The results can be easily extented when the sub-Gaussian assumption over $F-f^*$ is relaxed. It would only degrade the confidence in the main theorems (assuming for example that the class is sub-exponential). The conclusion would be similar. As long as the proportion of outliers is smaller than the rate of convergence, both ERM and RERM behave as if there was to contamination. 

\appendix

\section{Simulations}\label{sec:simu}
In this section, we present simple simulations to illustrate our theoretical findings. We consider regression problems in $\mathbb R^p$ both non-regularized and penalized with the $\ell_1$-norm. For $i=1,\cdots,N$, let us consider the following model:
\begin{equation*}
	Y_i = \inr{X_i,t^*} + \epsilon_i
\end{equation*}
where $(X_i)_{i=1}^N$ are i.i.d random variables distributed as $\mathcal N(0,I_p)$, $(\epsilon_i)_{i \in \mathcal I}$ are symmetric independent to $X$ random variables. Nothing is assumed on $(\epsilon_i)_{i \in \mathcal O}$. We consider different distributions for the noise $(\epsilon_i)_{i \in \mathcal I}$ . We consider
\begin{itemize}
	\setlength{\itemsep}{1pt}
	\setlength{\parskip}{0pt}
	\setlength{\parsep}{0pt}
	
	\item $\epsilon_i \sim \mathcal N(0,\sigma^2) $ Gaussian distribution.
	\item $ \epsilon_i \sim \mathcal T(2)$ Student distribution with 2-degree of freedom.
	\item $ \epsilon_i \sim \mathcal C (1)$ Cauchy distribution.
\end{itemize}
We study Huber's $M$ estimator defined as
$$ \hat{t}_N^{\gamma} \in \argmin_{t \in \mathbb R^p} \frac{1}{N} \sum_{i=1}^{N} \ell^{\gamma}(f(X_i),Y_i) $$
where $\ell^{\gamma}: \mathbb R \times \mathbb R  \mapsto \mathbb R^+$ is the Huber loss function defined as, $\gamma >0$, $u,y\in \mathbb R$, by
$$
\ell^{\gamma}(u,y) =  
\begin{cases}
\frac{1}{2}(y-u)^2&\text{ if }|u-y| \leq \gamma\\
\gamma|y-u|-\frac{\gamma^2}{2}&\text{ if }|u-y| > \gamma
\end{cases}
$$
Note that other loss functions could be considered as the absolute loss function, or more generally, any quantile loss function. According to Theorem~\ref{thm:huber}, we have
$$
\| \hat{t}_N^{\gamma} -t^*\|_2 \leq c \gamma\bigg( \sqrt{ \frac{p}{N} } + \frac{|\mathcal O|}{N}  \bigg)
$$
where $c>0$ is an absolute constant. We add malicious outliers following a uniform distribution over $[-10^{-5}, 10^5 ]$. We expect to obtain an error rate proportional to the proportion of outliers $|\cO|/ N$. We ran our simulations with $N=1000$ and $p=50$. The only hyper-parameter of the problem is $\gamma$. For the sake of simplicity we took $\gamma = 1$ for all our simulations. We see on Figure~\ref{fig:erm} that no matter the noise, the error rate is proportional to the proportion of outliers which matches with our theoretical findings. 
\begin{figure}[h] 
	\center
	\includegraphics[scale=0.7]{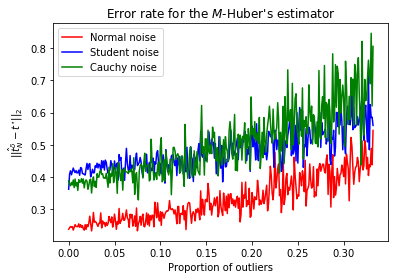}
	\caption{Error rate for the Huber's $M$-estimator ($p=50$ and $N = 1000$)}\label{fig:erm}
\end{figure}

In a second experiment, we study $\ell_1$ penalized $M$-Huber's estimator defined as
$$ \hat{t}_N^{\lambda, \gamma} \in \argmin_{t \in \mathbb R^p} \frac{1}{N} \sum_{i=1}^{N} \ell^{\gamma}(f(X_i),Y_i) + \lambda \|t\|_1$$
where $\ell^{\gamma}: \mathbb R \times \mathbb R  \mapsto \mathbb R^+$ is the Huber loss function and $\lambda > 0$ is a hyper-parameter. According to Theorem~\ref{thm:huber_pen} we have
$$
\| \hat{t}_N^{\gamma} -t^*\|_2 \leq c \gamma \bigg( \sqrt{ \frac{s \log(p)}{N} } + \frac{|\mathcal O|}{N}  \bigg)
$$
where $c > 0$ is an absolute constant. We ran our simulations with $N=1000$ and $p=1000$ and $s = 50$. The hyper-parameters of the problem are $\gamma$ and $\lambda$. For the sake of simplicity we take $\gamma = 1$ and $\lambda = 10^{-3}$ for all our simulations. We see on Figure~\ref{fig:rerm} that no matter the noise, the error rate is proportional to the proportion of outliers which matches our theoretical findings. The fact that the error rate may be large comes to the fact that we did not optimize the value of $\lambda$. 

\begin{figure}[h] 
	\center
	\includegraphics[scale=0.7]{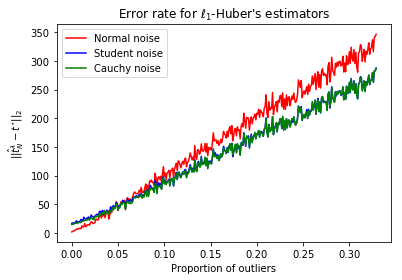}
	\caption{Error rate for $\ell_1$ penalized Huber's $M$-estimator ($p=1000$ and $N = 1000$ and $s = 50$)}\label{fig:rerm}
\end{figure}

\section{Lower bound minimax risk in regression where only the labels are contaminated} \label{sec:minimax}
This section is built on the work~\cite{chen2018robust} where the authors establish a general minimax theory for the $\varepsilon$-contamination model defined as $P_{(\varepsilon,\theta,Q)} = (1-\varepsilon)P_{\theta} + \varepsilon Q$ given a general statistical experiment $\{P_{\theta}, \theta \in \Theta \}$. A proportion $\varepsilon$ of outliers with same the distribution $Q$ contaminate $P_{\theta}$. Given a loss function $L(\theta_1,\theta_2)$, the minimax rate for the class $\{P_{(\varepsilon,\theta,Q)}, \theta \in \Theta,Q \}$ depends on the modulus of continuity defined as:
\begin{equation} \label{modulus_continuity}
w(\varepsilon,\Theta) = \sup \bigg \{  L(\theta_1,\theta_2)  :  TV(P_{\theta_1}, P_{\theta_2}) \leq \frac{\varepsilon}{1 - \varepsilon}, \theta_1, \theta_2 \in \Theta  \bigg  \}
\end{equation}
where $TV(P_{\theta_1}, P_{\theta_2})$ denotes the total variation distance between $P_{\theta_1}$ and $P_{\theta_2}$ defined as $TV(P_{\theta_1}, P_{\theta_2}) = \sup_{A \in \cF} |P_{\theta_1}(A) - P_{\theta_2}(A)|$, for $\cF$ the $\sigma$-algebra onto which $P_{\theta_1}$ and $P_{\theta_2}$ are defined.

\begin{Theorem}[Theorem 5.1~\cite{chen2018robust}] \label{thm_gao} Suppose there is some $\mathcal M(0)$ such that for $\varepsilon = 0$
	\begin{equation} \label{eq_gao}
	\inf_{\hat \theta} \sup_{\theta \in \Theta} \sup_{Q} P_{(\varepsilon,\theta,Q)} \bigg( L(\theta,\hat \theta)  \geq \mathcal M(\varepsilon)  \bigg) \geq c
	\end{equation}
	holds. Then, for any $\varepsilon \in [0,1]$ \eqref{eq_gao} holds for $\mathcal M(\varepsilon) = c \big(\mathcal M(0) \vee w(\varepsilon,\Theta) \big)$.
\end{Theorem}
$w(\varepsilon,\Theta)$ is the price to pay in the minimax rate when a proportion $\varepsilon$ of the samples are contaminated. To illustrate Theorem~\ref{thm_gao}, let us consider the linear regression model:
\begin{equation*}
	Y_i = \inr{X_i,\theta} + \epsilon_i
\end{equation*}
where without contamination $X_i \sim \cN(0,\Sigma)$, $\epsilon_i \sim \cN(0,\sigma^2)$ are independent. In~\cite{chen2016general}, the authors consider a setting when both the design $X$ and the response variable in the model can be contaminated i.e $(X_1,Y_1),\cdots,(X_N,Y_N) \sim  (1-\varepsilon)P_{\theta} + \varepsilon Q$, with $P_{\theta} = P(X) P(Y|X)$, $P(X) = \cN(0,\Sigma)$ and $P(Y|X) = \cN(X^T\theta,\sigma^2)$. They establish that the minimax optimal risk over the class of $s$-sparse vectors for the metric $L(\theta_1,\theta_2) = \|\theta_1-\theta\|_2^2$ is given by 
\begin{equation*}
	\sigma^2 \bigg(\frac{s\log(p/s)}{N} \vee \varepsilon^2  \bigg).
\end{equation*}
The question of main interest in our setting is the following: does the minimax risk for regression problem in the $\varepsilon$-contamination model remain the same when only the labels are contaminated ? \\
The following theorem answers to the above question.
\begin{Theorem}\label{minimax_rate} 
	Let $\{P_{\theta} = P_{(X,Y)}^{\theta}$ with $ Y = f_{\theta}(X) + \epsilon, \theta \in \Theta \} $ be a statistical regression model. For any $\theta \in \Theta$, $\varepsilon \in[0,1]$ let 
	\begin{align*}
		\cP_{\theta,\varepsilon} = \bigg\{ & \big( (1-\varepsilon)P_{\theta} + \varepsilon Q_{\theta} \big)^{ \otimes_{i=1}^{N}}, P_{\theta} = P_{(X,Y)}^{\theta} \mbox{  with  } Y = f_{\theta}(X) + \epsilon \\
		&  Q_{\theta} = P_{(X,\tilde Y)}^{\theta} \mbox{  with  } \tilde Y = f_{\theta}(X) + \tilde \epsilon \bigg\}
	\end{align*}
	Suppose there is some $\mathcal M(0)$ such that for $\varepsilon = 0$
	\begin{equation} \label{eq_gao2}
	\inf_{\hat \theta} \sup_{R_{\theta,\varepsilon} \in 	\cP_{\theta,\varepsilon}, \theta \in \Theta} R_{\theta,\varepsilon} \bigg( L(\theta,\hat \theta)  \geq \mathcal M(\varepsilon)  \bigg) \geq c
	\end{equation}
	holds. Then For any $\varepsilon \in [0,1]$ \eqref{eq_gao2} holds for $\mathcal M(\varepsilon) = c \big( \mathcal M(0) \vee w(\varepsilon,\Theta) \big)$
\end{Theorem}
Theorem~\ref{minimax_rate} states that the minimax optimal rates for regression problems in the $\varepsilon$-contamination model are the same when
\begin{enumerate}
	\item  Both the design $X$ and the response variable $Y$ are contaminated.
	\item Only the response variable $Y$is contaminated. 
\end{enumerate} 
The proof is very similar as the one of Theorem~\ref{thm_gao}. We present it here, for the sake of completeness.
\begin{proof}
	The case when $\mathcal M(\varepsilon) = c \mathcal M(0)$ is straightforward. Thus, the goal is to lower bound with a constant the following quantity
	\begin{equation*}
		\inf_{\hat \theta} \sup_{R_{\theta,\varepsilon} \in 	\cP_{\theta,\varepsilon}, \theta \in \Theta} R_{\theta,\varepsilon} \bigg(  L(\theta, \hat \theta) \geq w(\varepsilon , \Theta)  \bigg)
	\end{equation*}
	We use Le Cam's method with two hypotheses. The first goal is to find $\theta_1, \theta_2$ such that $L(\theta_1,\theta_2) \geq w(\varepsilon, \Theta)$. To do so, let $\theta_1,\theta_2$ be solution of
	\begin{equation*}
		\max_{\theta_1,\theta_2 \in \Theta} L(\theta_1,\theta_2) \quad \mbox{s.t} \quad TV(P_{\theta_1}, P_{\theta_2})  = TV(P_{(X,Y)}^{\theta_1},P_{(X,Y)}^{\theta_2})\leq \frac{\varepsilon}{1-\varepsilon}
	\end{equation*}
	Thus there exists $\varepsilon' \leq \varepsilon$ such that $TV(P_{\theta_1},P_{\theta_2}) = \varepsilon' /(1-\varepsilon')$ and $L(\theta_1,\theta_2) = w(\varepsilon, \Theta)$. To conclude, it is enough to find two distributions $R_{\theta_1,\varepsilon}$ and $R_{\theta_2,\varepsilon}$ in $\cP_{\theta_1,\varepsilon}$ and $\cP_{\theta_2,\varepsilon}$ such that $R_{\theta_1,\varepsilon} = R_{\theta_2,\varepsilon}$. It would imply that $\theta_1$ and $\theta_2$ are not identifiable from the model and the Le Cam's method would complete the proof. \\
	
	For $i \in \{1,2\}$ let $p_{\theta_i}$ be a density function defined for all $(x,y) \in \cX \times \cY$ as
	\begin{equation}
	 p_{\theta_i}(x,y) = \frac{d P_{(X,Y)}^{\theta_i}}{d \big(  P_{(X,Y)}^{\theta_1} + P_{(X,Y)}^{\theta_2}  \big)} (x,y)
	\end{equation}
	By conditioning, it is possible to write $p_{\theta_i}(x,y) = p_X(x) p_{Y|X=x}^{\theta_i}(y)$. Let $R_{\theta_1,\varepsilon}$ and $R_{\theta_2,\varepsilon}$ defined respectively as
	\begin{equation*}
		R_{\theta_1,\varepsilon} = (1-\varepsilon') P^{\theta_1}_{(X,Y)} + \varepsilon' P^{\theta_1}_{(X, \tilde Y)}  \quad \mbox{and} \quad R_{\theta_2,\varepsilon} = (1-\varepsilon') P^{\theta_2}_{(X,Y)} + \varepsilon' P^{\theta_2}_{(X, \tilde Y)}
	\end{equation*}
	where $P^{\theta_1}_{(X, \tilde Y)}$ and $P^{\theta_2}_{(X, \tilde Y)}$ are defined by their density functions
	\begin{align*}
		\forall (x,y) \in \cX \times \cY, & \frac{d P^{\theta_1}_{(X, \tilde Y)}}{d \big(  P_{(X,Y)}^{\theta_1} + P_{(X,Y)}^{\theta_2}  \big)} (x,y) = \frac{\big(p_{\theta_2}(x,y)-p_{\theta_1}(x,y) \big)  \bI \{  p_{\theta_2}(x,y) \geq p_{\theta_1}(x,y) \}}{TV\big(P_{(X,Y)}^{\theta_1},P_{(X,Y)}^{\theta_2} \big)} \\
		&  \frac{d  P^{\theta_2}_{(X, \tilde Y)}}{d \big(  P_{(X,Y)}^{\theta_2} + P_{(X,Y)}^{\theta_1}  \big)} (x,y) = \frac{\big(p_{\theta_1}(x,y)-p_{\theta_2}(x,y) \big)  \bI \{  p_{\theta_1}(x,y) \geq p_{\theta_2}(x,y) \}}{TV \big(P_{(X,Y)}^{\theta_1},P_{(X,Y)}^{\theta_2} \big)}
	\end{align*}
	Using Scheffé's theorem, it is easy to see that $P^{\theta_1}_{(X, \tilde Y)}$ and $P^{\theta_2}_{(X, \tilde Y)}$ are probability measures. Moreover, from the facts that $p_{\theta_i}(x,y) = p_X(x) p_{Y|X=x}^{\theta_i}(y)$, $\varepsilon' \leq \varepsilon$ and Lemma 7.2 in \cite{chen2018robust} we have $R_{\theta_1,\varepsilon} \in \cP_{\theta_1,\varepsilon}$ and $R_{\theta_2,\varepsilon} \in\cP_{\theta_2,\varepsilon}$. \\
	For any $(x,y) \in \cX \times \cY$. Straightforward computations give
	\begin{align*}
		\frac{dR_{\theta_1,\varepsilon}}{d \big(  P_{(X,Y)}^{\theta_1} + P_{(X,Y)}^{\theta_2}  \big)}(x,y) & = (1-\varepsilon') p_{\theta_1}(x,y) + \varepsilon'  \frac{\big(p_{\theta_2}(x,y)-p_{\theta_1}(x,y) \big)  \bI \{  p_{\theta_2}(x,y) \geq p_{\theta_1}(x,y) \}}{TV\big(P_{(X,Y)}^{\theta_1},P_{(X,Y)}^{\theta_2} \big)}\\ 
		&= (1-\varepsilon') p_{\theta_1}(x,y) + \varepsilon'  \frac{\big(p_{\theta_2}(x,y)-p_{\theta_1}(x,y) \big)  \bI \{  p_{\theta_2}(x,y) \geq p_{\theta_1}(x,y) \}}{\varepsilon'/(1-\varepsilon')}\\ 
		&= (1-\varepsilon') \big( p_{\theta_1}(x,y) + (p_{\theta_2}(x,y)-p_{\theta_1}(x,y))  \bI \{  p_{\theta_2}(x,y) \geq p_{\theta_1}(x,y) \}     \big) \\
		&= (1-\varepsilon') \big( p_{\theta_2}(x,y) + (p_{\theta_1}(x,y)-p_{\theta_2}(x,y))  \bI \{  p_{\theta_1}(x,y) \geq p_{\theta_2}(x,y) \}     \big) \\
		& = 	\frac{dR_{\theta_2,\varepsilon}}{d \big(  P_{(X,Y)}^{\theta_1} + P_{(X,Y)}^{\theta_2}  \big)}(x,y) 
	\end{align*}
\end{proof}

\section{$\ell_1$-penalized Huber's M-estimator with non-isotropic design} \label{RE_huber}
In this section, we relax the isotropic assumption on the design $X$. Recall that a random variable $X$ is isotropic if for every $t \in \bR^p$, $\bE \inr{X,t}^2 = \|t\|_2^2$. Instead, we consider covariance matrices satisfying a Restricted Eigenvalue condition (RE). A matrix $\Sigma$ is said to satisfy the restricted eigenvalue condition RE$(s, c_0)$ with some constant $\kappa > 0$, if $\|\Sigma^{1/2} v \|_2  \geq \kappa \|v_J\|_2$ for any vector $v$ in $\bR^p$ and any set $J \subset \{1,\cdots,p \}$ such that $|J| \leq s$ and $\|v_{J^c}\|_1 \leq c_0 \|v_J\|_1$. We want to derive a result similar to Theorem~\ref{thm:huber_pen} when $X \sim \mathcal ~N(0,\Sigma)$, for $\Sigma$ satisfying RE$(s,c)$ for $c$ an absolute constant. With non isotropic design we cannot use Lemma~\ref{lemma_lasso} and the computation of the Gaussian mean-width is more involved. 
\begin{Lemma}\label{lemma_lasso_RE}
	Let $\bB_1^p$ denote the unit ball induced by $\|\cdot\|_1$. Let us assume that the design $X$ has a covariance matrix satisfying RE$(s,9)$ with constant $\kappa >0$. If the \textit{oracle} $t^{*}$ is $s$-sparse and $100s \leq \big( \kappa \rho / r   \big)^2$ then:
	\begin{equation*}
		\Delta(A,\rho,\delta) = \inf_{w \in H_{A,\rho,\delta}} \sup_{z^*\in\Gamma_{t^*}(\rho)}\inr{z^*, w} \geq 4\rho/5 \enspace.
	\end{equation*}
\end{Lemma}
The difference with Lemma~\ref{lemma_lasso} is the term $\kappa$ coming from the $\textnormal{RE}(s,9)$ condition.
\begin{proof}
	The goal is to find $\rho$ such that $\Delta(A,\rho,\delta) \geq (4/5)\rho$. Recall that 
	\begin{equation}
	(\partial \norm{\cdot})_t=\left\{
	\begin{array}{cc}
	\{z^*\in \bS^*:\inr{z^*, t}=\norm{t}\} & \mbox{ if } t\neq0\\
	\bB^* & \mbox{ if } t=0
	\end{array}
	\right.\enspace.
	\end{equation}
	Since $F = \{ \inr{t,\cdot}, t \in \bR^p   \}$, $\|f\|_{L_2} = \| \inr{t,X}\|_{L_2} = \| \Sigma^{1/2} t \|_2$. Let $w$ be in $\bR^p$ such that $\|w\|_1 = \rho$ and $\|\Sigma^{1/2} w \|_2 \leq r$. Let us denote by $I$ the support of $t^*$ and $P_I w$ the projection of $w$ on $(e_i)_{i \in I}$. By assumption we have $|I| \leq s$. Let $z$ in $	(\partial \norm{\cdot})_{t^*}$such that for every $i \in I$, $z_i = \textnormal{sign}(t^*_i)$, and for every $i \in I^c$, $z_I = \textnormal{sign}(w_i)$. It is clear that $z$ is norming for $t^*$ i.e $\inr{z,t^*} = \|t^*\|_1$, $z \in \bS_1^* = \bS_{\infty}$ and
	\begin{align*}
		\inr{z,w} =  \inr{z , P_I w} + \|P_{I^c}\|_1 \geq - \|P_I w\|_1 + \|P_{I^c}\|_1  = \rho -2  \|P_I w\|_1 
	\end{align*}
	Let us assume that $P_I w$ satisfies $\|P_{I^c} w \|_1 > 9 \|P_I w \|_1$ which can be rewritten as $\rho \geq 10 \|P_I w\|_1$. It follows that 
	\begin{align*}
		\inr{z,w} \geq  \rho -2  \|P_I w\|_1  \geq \rho - \frac{1}{5} \rho \geq 4\rho/5,
	\end{align*}
	and the sparsity equation is satisfied. Now let us turn to the case when $\|P_{I^c} w \| \leq 9 \|P_I w \|_1$. From the RE$(s,9)$ condition we have $\| P_I w \|_2 \leq \| \Sigma^{1/2} w\|_2/\kappa$ and it follows 
	\begin{align*}
		\rho -2  \|P_I w\|_1  \geq \rho - 2\sqrt{s}\|P_I w \|_2  \geq  \rho - \frac{2}{\kappa} \sqrt{s }\| \Sigma^{1/2} w \|_2  \geq  \rho - \frac{2}{\kappa} \sqrt{s } r \geq 4\rho/5
	\end{align*}
\end{proof}
Now, let us turn to the computation of the Gaussian mean-width when the design $X$ is not isotropic. To do so, we use the following Proposition.
\begin{Proposition}[Proposition 1~\cite{bellec2019localized}] \label{bellec_prop}
	Let $p \geq 1$ and $M\geq 2$. Let $T$ be the convex hull of $M$ points in $\bR^p$ and assume that $T \subset \bB_2^p$. Let $\mathbf{G} \sim \cN(0,I_p)$. Then for all $s >0$,
	\begin{equation*}
		\bE \sup_{t \in s \bB_2^p \cap T} \inr{t,\mathbf{G}} = w(s\bB_2^p \cap T) \leq 4\sqrt{  \log_+(4eM(s^2 \wedge 1)) } ,
	\end{equation*}
	where $\log_+(a) = \max(1,\log(a))$.
\end{Proposition}
When $F = \{ \inr{t,\cdot}, t \in \bR^p   \}$ and the covariance matrix of $X$ is $\Sigma$, for every $r,\rho >0$ we have
\begin{equation*}
	w \big( F \cap (f^* + r\bB_{2} \cap \rho \bB_1^p \big) = w \big( r \bB_2^p \cap \rho \Sigma^{1/2} \bB_1^p  \big) = w \big( r \bB_2^p \cap \rho T \big)
\end{equation*}
where $T := \Sigma^{1/2} \bB_1^p$ is the convex hull of $(\pm \Sigma^{1/2} e_i)_{i=1}^p$. To apply Proposition~\ref{bellec_prop} it is necessary to assume that for every $i=1,\cdots,p$, $\Sigma^{1/2} e_i \in \bB_2^p$ which holds when $\Sigma_{i,i} \leq 1$ and we obtain the following result:
\begin{Proposition} \label{prop_non_isotropic}
	Let $F = \{ \inr{t,\cdot}, t \in \bR^p   \}$ and assume that $\Sigma$, the covariance matrix of $X$, satisfies $\Sigma_{i,i} \leq 1$ for every $i=1,\cdots,p$. Then, for every $r,\rho >0 $
	\begin{equation*}
		w \big( F \cap (f^* + r\bB_{2} \cap \rho \bB_1^p \big) \leq 4 \rho \sqrt{  \log_+(8ep((r/\rho)^2 \wedge 1)) } 
	\end{equation*}
\end{Proposition}
By taking step by step the computations from Section~\ref{app:huber_regu} we obtain the following theorem extending Theorem~\ref{thm:huber_pen} for a non-isotropic design:
\begin{Theorem} \label{thm:huber_pen_non_isotropic}
	Let $\cI \cup \cO$ denote a partition of $\{1, \cdots, N \}$ such that $|\cI|\geq |\cO|$ and $(X_i,Y_i)_{i= 1}^N$ be random variables valued in $\bR^p \times \bR$ such that $(X_i)_{i=1}^N$ are i.i.d random variable with $X_1 \sim \cN(0,\Sigma)$, where $\Sigma$ satisfies $\Sigma_{i,i} \leq 1$ for $i=1,\cdots,p$ and verifies RE$(s,9)$ for some constant $\kappa >0$. Assume that for all $i \in \{ 1,\cdots, N \}$ 
	\begin{equation*}
		Y_i = \inr{X_i, t^*} + \epsilon_i \enspace,
	\end{equation*}
	where $t^*$ is $s$-sparse. Let 
	\[
	\tilde r(\alpha,\delta) = c \frac{\gamma}{\alpha } \bigg(  \sqrt{\frac{s\log(p)}{\kappa^2 N} }\vee \sqrt \frac{\log(1/\delta)}{N} \vee \frac{|\cO|}{N}  \bigg)
	\]
	Assume that $(\epsilon_i)_{i \in I}$ are i.i.d random variables independent to $(X_i)_{i \in \cI}$ such that there exists $\alpha>0$ such that
	\begin{equation} \label{Bernstein_huber}
	F_{\epsilon} \big(\gamma -    c \tilde r(\alpha,\delta) \big) - F_{\epsilon} \big(  c \tilde r(\alpha,\delta)   - \gamma  \big)\geq \alpha 
	\end{equation}
	where $	F_{\epsilon} $ denotes the cdf of $\epsilon$ where $\epsilon$ is distributed as $\epsilon_i$, for $i$ in $\cI$. Set
	\begin{equation*}
\lambda = c \gamma \bigg(  \sqrt{\frac{\log(p)}{N}} \vee \kappa \sqrt \frac{\log(1/\delta)}{s N} \vee\frac{\kappa |\cO| }{\sqrt sN}  \bigg) \enspace.
\end{equation*}
	Then with probability larger than $1-\delta$ the estimator $\hat{t}_N^{\gamma,\lambda}$ defined in Equation~\eqref{rerm_huber} satisfies
	\begin{align*}
&		\|\hat{t}_N^{\gamma,\lambda}- t^*\|_{2} \leq \tilde r(\alpha,\delta) , \quad   P\cL_{\hat{t}_N^{\gamma,\lambda}} \leq c\alpha(\tilde r(\alpha,\delta))^2  \\
& \mbox{and } \quad \|\hat{t}_N^{\delta,\lambda} - t^*\|_1 \leq c \frac{\gamma}{\kappa \alpha} \bigg( \frac{s}{\kappa} \sqrt \frac{\log(p)}{N} \vee \sqrt \frac{s \log(1/\delta)}{N} \vee  \sqrt s \frac{|\cO|}{N}\bigg) 
\end{align*}
\end{Theorem}
When $\epsilon_i \sim C(1)$, we can use the same argument as in Section~\ref{app:hub_reg}. Equation~\eqref{Bernstein_huber} holds with $\alpha = 1/4$ and $\gamma = 2\tan(\pi/8)$ if
\begin{equation} \label{eq:ber:hub}
 \sqrt{\frac{s\log(p)}{\kappa^2 N} }\vee \sqrt \frac{\log(1/\delta)}{N} \vee \frac{|\cO|}{N}   \leq c \enspace.
\end{equation}
Now, for the sake of comparizon with~\cite{dalalyan2019outlier}, let us consider the case where $\epsilon_i \sim \cN(0,\sigma^2)$. For any $t \in \bR$, $F_{\epsilon}(t) = \Phi (t/\sigma)$, where $\Phi$ denotes the cdf of a standard gaussian random variable. Thus, Equation~\eqref{Bernstein_huber} can be rewritten as
\[
\Phi \bigg( \frac{\gamma - c \tilde r(\alpha,\delta) }{\sigma} \bigg) \geq \frac{1+\alpha}{2} \enspace,
\]
which is verified for $\gamma = c \sigma$ and $\alpha = 1/4 $ if Equation~\eqref{eq:ber:hub} holds. We improve the main result of~\cite{dalalyan2019outlier} in two aspects: 
\begin{enumerate}
	\item We obtain the error rate
	\[
     \sigma \bigg(  \sqrt{\frac{s\log(p)}{\kappa^2 N} }\vee \sqrt \frac{\log(1/\delta)}{N} \vee \frac{|\cO|}{N}  \bigg) \enspace, 
	\]
	while their rate is 
	\[
	 \sigma \bigg(  \sqrt{\frac{s\log(p/\delta)}{\kappa^2 N} }\vee \frac{|\cO| \log(n/\delta)}{N}  \bigg) \enspace. 
	\]
	In our rate, the term $\kappa$ is only in factor with the first term and not the probability confidence. The extra term $\log(n/\delta)$ can be problematic in their bound when one wants to obtain exponentially large confidence. 
	\item Their bound holds for $|\cO| \leq cN / \log(N)$ while ours holds for $|\cO| \leq cN$.
\end{enumerate}
 However, note that when $|\cO| \geq \sqrt{s\log(p)N}/\kappa$, our regularization parameter depends on the unknown sparsity and the number of outliers, which is not the case in~\cite{dalalyan2019outlier}. 

\begin{Remark}
	It is possible to replace $\log(p)$ by $\log(p/s)$ and recover the exact minimax rate of convergence. However, the price to pay is that the regularization parameter $\lambda$ would always depend on the sparsity $s$, even when the number or outliers in small.  
\end{Remark}	

\section{Proofs main Theorems} \label{sec:proof_main_thm}

\subsection{Proof Theorem~\ref{thm:main} in the sub-Gaussian setting} \label{sec:proof1}
Let $r(\cdot,\delta)$ be such that for all $A>0$
\[
r(A,\delta) \geq  c \bigg( r_{\cI}^{SG} (A) \vee AL \sqrt \frac{\log(1/\delta)}{N} \vee AL \frac{|\mathcal O |}{N} \bigg) \enspace.
\]
Moreover, let $A$ satisfying assumption~\ref{assum:fast_rates} with $ r(\cdot,\delta)$.\\

The proof is split into two parts. First we identify a stochastic argument holding with large probability. Then we show on that event that $\|\hat{f}_N -f^* \|_{L_2} \leq r(A,\delta)$. Finally, at the very end of the proof we show that $P\cL_{\hat f_N} \leq  r^2(A,\delta)/A$. 

\paragraph{Stochastic arguments}
First we identify the stochastic event onto which the proof easily follows. Let $\cF_r =  F \cap \big( f^* +  r(A,\delta) \bB_2 \big)$ and define
\begin{align} \label{omega_i}
	& \Omega_{\cI} =\bigg\{  \sup_{f \in \cF_r}  \big| \big(P-P_{\cI}  \big) \big( \ell_f - \ell_{f^*}   \big)  \big|  \leq c\frac{L}{\sqrt{| \cI |}} \bigg( w(\cF_r) + r(A,\delta)  \sqrt{\log(1/\delta)}  \bigg) \bigg\}  \\
	\label{omega_o}
	&\Omega_{\cO} = \bigg\{	\sup_{f \in \cF_r} \big| \big(P-P_{\cO}  \big) |f- f^*|  \big|  \leq \frac{c}{\sqrt{| \cO |}} \bigg( w(\cF_r) + r(A,\delta)  \sqrt{\log(1/\delta)}  \bigg) \bigg\} 
\end{align}
where for any $K \subset \{1,\cdots, N \}$, $g: \cX \times \cY \mapsto \bR$, $P_K g = 1/(|K|)\sum_{i \in K} g(X_i,Y_i)$ and $w(\cF_r)$ is the Gaussian mean-width of $\cF_r$ .  Finally let us define $\Omega = \Omega_{\cI} \cap \Omega_{\cO}$. 

\begin{Lemma} \label{lem:sto_arg}
	Grant Assumptions~\ref{assum:distri},~\ref{assum:lip_conv},~\ref{assum:convex} and assume that $F-f^*$ is $1$-sub-Gaussian. The event $\Omega$ holds with probability larger than $1-\delta$
\end{Lemma}

The proof of Lemma~\ref{lem:sto_arg} necessitates several tools from sub-Gaussian random variables that we introduce now.\\
Let $\psi_2(u)= \exp(u^2)-1$. The Orlicz space $L_{\psi_2}$ associated to $\psi_2$ is defined as the set of all random variables Z on a probability space $(\Omega, \cA, \bP)$ such that $\|Z\|_{\psi_2} < \infty$ where 
\begin{equation*} 
	\|Z\|_{\psi_2} = \inf   \{ c>0,  \bE \psi_2 \bigg(  \frac{Z}{c} \bigg) \leq1  \}
\end{equation*}
Let $H \subset L_2$ and $(X_h)_{h \in H}$ be a stochastic process indexed by the metric space $(H,\| \cdot\|_{L_2})$ satisfying the following Lipschitz condition
\begin{equation} \label{eq:lip}
\mbox{for all } h,g \in H, \quad \|X_g-X_h\|_{\psi_2} \leq \| g-h \|_{L_2} \enspace.
\end{equation}
For such a process it is possible to control the deviation of $\sup_{h \in H} X_h$ in terms of the geometry of $(H,\| \cdot\|_{L_2})$ through the Gaussian mean-width of $H$. 
\begin{Theorem}[\cite{ledoux2013probability}, Theorem 11.13] \label{thm_generic_chaining}
	Let $(X_h)_{h\in H}$ be a random process indexed by $(H,\|\cdot\|_{L_2})$ satisfying Equation~\eqref{eq:lip}. Then, there exists an absolute constant $c>0$ such that for all $u >0$
	\begin{equation*}
		\bP  \bigg(  \sup_{h,g \in H} |X_h-X_g| \geq c ( w(H)+u D_{L_2}(H)  \bigg)  \leq \exp(-u^2)
	\end{equation*}
	where $w(H)$ is the Gaussian mean width of $H$ and $D_{L_2}(H)$ is the $L_2$-diameter. 
\end{Theorem}

The following Lemma allows to control the $\psi_2$-norm of a sum of independent centered random variables. 
\begin{Lemma}[\cite{chafai2012interactions}, Theorem 1.2.1] \label{lem:chafai1}
	Let $X_1,\cdots,X_N$ be independent real random variables such that for all $i=1,\cdots,N$, $\bE X_i = 0$. Then
	\begin{equation*}
		\| \sum_{i=1}^N X_i \|_{\psi_2} \leq 16 \bigg(  \sum_{i=1}^N  \|X_i\|_{\psi_2}^2   \bigg)^{1/2}
	\end{equation*}
\end{Lemma}
The following Lemma connects  $\psi_2$-bounded random variable with the control of its Laplace transform. 
\begin{Lemma}[\cite{chafai2012interactions}, Theorem 1.1.5] \label{lem:chafai2}
	Let $Z$ be a real valued random variable. The following assertions are equivalent
	\begin{itemize}
		\item There exists $K >0$ such that $\|Z\|_{\psi_2} \leq K$
		\item There exist absolute constants $c_1,c_2,c_3>0$ such that for every $\lambda \geq c_1/K $
		\begin{equation}
		\bE \exp(\lambda |Z|) \leq c_3\exp(c_2 \lambda^2 K^2 )
		\end{equation}
	\end{itemize}
\end{Lemma}
We are now in position to prove Lemma~\ref{lem:sto_arg}. 
\begin{proof}
	First we prove that $\Omega_{\cI}$ holds with probability larger than $1-\delta/2$. Let us assume that for any $f,g$ in $\cF_r$, the following condition holds
	\begin{equation} \label{proof:ri}
	\| \big(P-P_{\cI}  \big) \big( \ell_f - \ell_{g}   \big)  \|_{\psi_2}  \leq c(L/\sqrt{|\cI|}) \|f-g\|_{L_2} \enspace .
	\end{equation}
	Then, from Theorem~\ref{thm_generic_chaining}, there exists an absolute constant $c>0$ such that with probability larger than $1-\delta/2$
	\begin{align*}
		\sup_{f \in \cF_r} \bigg | \big(P-P_{\cI}  \big) \big( \ell_f - \ell_{f^*}   \big)  \bigg |  & \leq \sup_{f,g \in \cF_r} \bigg | \big(P-P_{\cI}  \big) \big( \ell_f - \ell_{g}   \big)  \bigg |  \\
		&  \leq  c\frac{L}{\sqrt{|\cI|}} \big( w(\cF_r) + \sqrt{\log(1/\delta)} D_{L_2}( \cF_r ) \big) \\
		& \leq  c\frac{L}{\sqrt{|\cI|}} \big(  w(\cF_r)  + \sqrt{\log(1/\delta)} r(A,\delta) \big) \enspace,
	\end{align*}
	concluding the proof for $\Omega_{\cI}$. Since $(X_o)_{o \in \cO}$ are i.i.d as $\mu$, with the same reasoning if we assume that
	\begin{equation} \label{proof:ro}
	\big\| \big(P-P_{\cO}  \big) |f - g|   \big\|_{\psi_2}  \leq (c /\sqrt{|\cO|}) \|f-g\|_{L_2} \enspace,
	\end{equation}
	then, with probability larger than $1 - \delta/2$:
	\begin{equation*}
		\sup_{ f \in \cF_r} \bigg| \big(P-P_{\cO}  \big) |f- f^*|  \bigg|  \leq \frac{c}{\sqrt{|\cO|}}   \big( w(\cF_r) + \sqrt{\log(1/\delta)} r(A,\delta) \big) \enspace.
	\end{equation*}
	
	Thus, to finish the proof it remains to show that Equations~\eqref{proof:ri} and~\eqref{proof:ro} hold. From Lemma~\ref{lem:chafai1} we get
	\begin{align*}
		\| \big(P-P_{\cI}  \big) (\ell_f - \ell_g)   \|_{\psi_2} & \leq 16 \bigg(  \sum_{i \in \cI}  \frac{\| (\ell_f- \ell_g)(X_i,Y_i) -\bE  (\ell_f- \ell_g)(X_i,Y_i)  \|_{\psi_2}^2 }{|\cI|^2} \bigg)^{1/2} \\
		& = \frac{16}{\sqrt{|\cI|}} \| (\ell_f- \ell_g)(X,Y) -\bE  (\ell_f- \ell_g)(X,Y)  \|_{\psi_2}
	\end{align*}	
	Thus, it remains to show that $\| (\ell_f- \ell_g)(X,Y) -\bE  (\ell_f- \ell_g)(X,Y)  \|_{\psi_2} \leq c L \|f-g\|_{L_2}$ for $c>0$ an absolute constant. To do so, we use Lemma~\ref{lem:chafai2}. Let $\lambda \geq cL/(\|f-g\|_{L_2})$. From the symmetrization principle (Lemma 6.3 in \cite{ledoux2013probability}) and the contraction principle (Theorem 2.2 in \cite{MR2829871}) we get
	\begin{align*}
		\bE \exp(\lambda | (\ell_f- \ell_g)(X,Y) -\bE  (\ell_f- \ell_g)(X,Y)  |) & \leq \bE \exp(2 \lambda \sigma  (\ell_f- \ell_g)(X,Y)  ) \\
		& \leq \bE \exp(4L \lambda \sigma  (f- g)(X)  ) \\
		& \leq \bE \exp(4L \lambda |f- g|(X)) 
	\end{align*}
	where $\sigma$ is a Rademacher random variable, independent to $(X,Y)$. From the sub-Gaussian Assumption we get
	\begin{equation*}
		\bE \exp(\lambda | (\ell_f- \ell_g)(X,Y) -\bE  (\ell_f- \ell_g)(X,Y)  |) \leq \bE \exp(16^2 \lambda^2 L^2 \|f-g\|_{L_2}^2)
	\end{equation*}
	which concludes the proof for $\Omega_{\cI}$ with Lemma~\ref{lem:chafai2}. For $\Omega_{\cO}$, we apply the same reasoning without the contraction step. 
\end{proof}	

\paragraph{Deterministic argument}
In this paragraph we place ourselves on the event $\Omega = \Omega_{\cI} \cap \Omega_{\cO}$. The main argument uses the convexity of the class $F$ with the one of the loss function.\\
From the definition of $\hat{f}_N$, we have $P_N \cL_{\hat{f}_N} \leq 0$. To show that $F \cap (f^*+ r(A,\delta) \bB_2) $ it is sufficient to show that for all functions $f \in F$ such that $F \backslash (f^*+ r(A,\delta) \bB_2) $ we have $P_N \cL_f >0$. Let $f$ in $F \backslash (f^*+ r(A,\delta) \bB_2)$ . By convexity of $F$ there exists a function $f_1 \in f^* + r(A,\delta) \bS_2$ for which
\begin{equation*}
	f-f^* = \alpha(f_1-f^*)
\end{equation*}
where $ \alpha = \big( \|f-f^*\|_{L_2}/r(A,\delta)\big)  \geq 1$. For all $i \in \{1,\cdots,N \}$, let $\psi_i: \mathbb R \rightarrow \mathbb R $ be defined for all $u\in \R$ by 
\begin{equation*}
	\psi_i(u) = \ell (u + f^{*}(X_i), Y_i) - \ell (f^{*}(X_i), Y_i).
\end{equation*}
The functions $\psi_i$ are such that $\psi_i(0) = 0$, they are convex under assumption~\ref{assum:lip_conv}. In particular $\alpha \psi_i(u) \leq \psi_i(\alpha u)$ for all $u\in\mathbb R$ and $\alpha \geq 1$ and $\psi_i(f(X_i) - f^{*}(X_i) )= \ell (f(X_i), Y_i) - \ell (f^{*}(X_i), Y_i) $ so that the following holds:
\begin{align*}
	\nonumber P_N \cL_f & = \frac{1}{N} \sum_{i=1}^{N}  \psi_i \big( f(X_i)- f^{*}(X_i) \big)= \frac{1}{N} \sum_{i=1}^{N}  \psi_i(\alpha( f_1(X_i)- f^{*}(X_i) ))\\
	&\geq \frac{\alpha}{N} \sum_{i=1}^{N}   \psi_i( f_1(X_i)- f^{*}(X_i)) = \alpha P_N \cL_{f_1}.
\end{align*}
From the previous argument it follows that $P_N \cL_f \geq \alpha P_N \cL_{f_1}$. Therefore it is enough to show that $P_N \cL_{f_1} >0 $ for every $f_1 \in F\cap (f^* +r(A,\delta) \bS_{2})$. We have 
\begin{align*}
	P_N \cL_{f_1} = \frac{|\cI|}{N} P_{\cI} \cL_{f_1} + \frac{|\cO|}{N} P_{\cO} \cL_{f_1}
\end{align*}
On $\Omega_{\cI}$ (see Equation~\eqref{omega_i}) it follows that
\begin{align*}
	P_{\cI} \cL_{f_1} \geq P \cL_{f_1} - \frac{cL}{\sqrt{|\cI|}} \bigg( w(\cF_r) + \sqrt{\log(1/\delta)} r(A,\delta)  \bigg)
\end{align*}
From assumption~\ref{assum:fast_rates} and the definition $r(A,\delta)$ it follows that
\[
P_{\cI} \cL_{f_1} \geq \frac{r^2(A,\delta)}{A} - \frac{cL}{\sqrt{| \cI|}} \bigg( \frac{  \sqrt{| \cI |} r^2(A,\delta)}{AL}  + \sqrt{\log(1/\delta)} r(A,\delta)  \bigg) = c \bigg( \frac{r^2(A,\delta)}{A} -  L r(A,\delta) \sqrt \frac{\log(1/\delta)}{N} \bigg)
\]
From assumption~\ref{assum:lip_conv}, it follows that
\begin{align*}
	 P_{\cO}\cL_{f_1} \geq - P_{\cO} | \ell_{f_1} - \ell_{f^*} | \geq -  L P_{\cO} | f_1 - f^* |   \enspace.
\end{align*}
On $\Omega_{\cO}$ (see Equation~\eqref{omega_o}), we get
\begin{align*} 
	 \frac{ |\cO| }{N}  P_{\cO} \cL_{f_1} & \geq - L \frac{ |\cO| }{N} \|f_1-f^*\|_{L_1} -  \frac{c \sqrt{| \cO| }}{N} \bigg(  w(\cF_r) + r(A,\delta) \sqrt{\log(1/\delta)}   \bigg)   \\
	& \geq  -  L \frac{ |\cO| }{N} r(A,\delta) -  c \bigg( \frac{r^2(A,\delta)}{A}  + L r(A,\delta) \sqrt \frac{\log(1/\delta)}{N}  \bigg) \enspace.
\end{align*}	
Finally, from the definition of $r(A,\delta)$, we obtain
\begin{align*}
	P_N \cL_{f_1} & \geq c \bigg( \frac{r^2(A,\delta)}{A} - Lr(A,\delta)  \sqrt \frac{\log(1/\delta)}{N}  -  L r(A,\delta) \frac{| \cO | }{N}\bigg)  >0 \enspace,
\end{align*}
which concludes the proof for the error rate.\\
We finish the proof by establishing the result for the excess risk. Since $\|\hat f_N - f^* \|_{L_2} \leq r(A,\delta)$, on $\Omega_{\cI}$ we have
\begin{align*}
	P\cL_{\hat{f}_N} & \leq P_{\cI}  \cL_{\hat{f}_N} +c \bigg(  \frac{r^2(A,\delta)}{A} + L r(A,\delta)   \sqrt \frac{\log(1/\delta)}{N} \bigg)  \\
	& = \frac{N}{|\cI|} P_N \cL_{\hat{f}_N} - \frac{|\cO|}{|\cI|} P_{\cO} \cL_{\hat{f}_N} +c \bigg(  \frac{r^2(A,\delta)}{A} + L r(A,\delta)   \sqrt \frac{\log(1/\delta)}{N} \bigg)\\
	& \leq  -\frac{|\cO|}{|\cI|} P_{\cO} \cL_{\hat{f}_N} +c \bigg(  \frac{r^2(A,\delta)}{A} + L r(A,\delta)   \sqrt \frac{\log(1/\delta)}{N} \bigg)  \enspace.
\end{align*}
On $\Omega_{\cO}$ we have 
\[
 L\frac{|\cO|}{|\cI|} P_{\cO}|\hat{f}_N-f^*| \leq  L\frac{|\cO|}{|\cI|} r(A,\delta) + c \bigg( \frac{r^2(A,\delta)}{A}  + L r(A,\delta) \sqrt \frac{\log(1/\delta)}{N} \bigg)  \enspace.
 \]
Finally, from the definition of $r(A,\delta)$ we obtain 
\[
P\cL_{\hat{f}_N} \leq c\frac{r^2(A,\delta)}{A}
\]

\subsection{Proof Theorem~\ref{thm:main} in the local bounded framework} \label{sec:proof2_bounded}
The deterministic argument is exactly the same as for the sub-Gaussian case. Recall that $r^b(A,\cdot)$ is defined as
\[
r(A,\delta) \geq  c \bigg( r_{\cI}^B (A) \vee AL \sqrt \frac{\log(1/\delta)}{N} \vee AL \frac{|\mathcal O |}{N} \bigg) \enspace,
\]
and satisfies Assumption~\ref{assum:fast_rates} for $A > 0$. Recall that $\cF_r = F \cap (f^* + r(A,\delta) \bB_2)$. 
To study the stochastic argument we use the following result:
\begin{Theorem}[Theorem 2.6,~\cite{koltchinskii2011empirical}] \label{thm:bousquet}
	Let $\cF$ be a class of functions bounded by $M$. For all $t>0$, with probability larger than $1-\exp(-t)$
	\begin{equation*}
		\sup_{f \in \cF} |(P_N-P)f| \leq \bE \sup_{f \in \cF} |(P_N-P)f| + \sqrt{2\frac{t}{N} \bigg( \sup_{f \in \cF} Pf^2 +  2M \bE \sup_{f \in \cF} |(P_N-P)f| \bigg)} + \frac{tM}{N} 
	\end{equation*}
\end{Theorem}
Let us define: 
\begin{align*}
	& \Omega_{\cI} := \bigg \{  \sup_{f \in \cF_r}  \big|(P-P_{\cI})\cL_f\ \big| \leq   c  \frac{r^2(A,\delta)}{A}  \bigg \}  \\
	&  \Omega_{\cO}  :=  \bigg \{  \sup_{f \in \cF_r}  \big|(P-P_{\cO})|f - f^*| \big| \leq   c\frac{|\cI |}{ | \cO |} \frac{r^2(A,\delta)}{AL}  \bigg \}
\end{align*}
\begin{Lemma}
	Grant Assumptions~\ref{assum:distri},~\ref{assum:lip_conv} and~\ref{assum:convex}. Assume than for all $f \in F \cap (f^* +  r(A,\delta) \bB_2 )$ and $x \in \cX: 	|f(x)-f^*(x)| \leq 1$. Then, the event $\Omega = \Omega_{\cI} \cup \Omega_{\cO}$ holds with probability larger than $1-\delta$.
\end{Lemma}
\begin{proof}
	Let $(\sigma_i)_{i \in \cI}$ be i.i.d Rademacher random variables independent to $(X_i)_{i \in \cI}$, from the symmetrization and contraction Lemmas (see~\cite{ledoux2013probability})
	\begin{align*}
		\bE & \sup_{f  \in \cF_r} |(P_{\cI}-P)\cL_f|  \leq 4L  \bE \sup_{f \in \cF }  \enspace \frac{1}{|\cI|} \sum_{i \in \cI} \sigma_i (f-f^*)(X_i) \leq c \frac{r^2(A,\delta)}{A}
	\end{align*}
	where we used the of $r_{\cI}^B(\cdot)$ and the fact that $r^b(A,\delta)\geq r_{\cI}^B(A)$ for all $A >0$. Under the local bounded assumption, any function $f$ in $\cF_r$,  $|\cL_f(x,y)|  \leq L$ for all $(x,y) \in \cX \times \cY$.   For any $t>0$, it follows from Theorem~\ref{thm:bousquet} that for any function $f$ in $\cF_r$
	\begin{align*}
		|(P_{\cI}-P)\cL_f | & \leq c \bigg[ \frac{(r^b(A,\delta))^2}{A} + \frac{L \log(1/\delta)}{N} +  L \sqrt{\frac{\log(1/\delta)}{|\cI|} \bigg( (r^b( A,\delta))^2 +   \frac{(r^b( A,\delta))^2}{AL}   \bigg)} \bigg]   \\
		& \leq c  \frac{(r^b(A,\delta))^2}{A}  \enspace. 
	\end{align*}
	The proof for $\Omega_{\cO}$ uses exactly the same arguments. 
\end{proof}

\subsection{Proof Theorem~\ref{thm:reg} in the sub-Gaussian framework} \label{sec:proof2}
Recall that $\tilde r(A,\rho^*,\delta)$ is such that: 
\[
\tilde r(A,\rho^*, \delta) \geq \tilde  r_{\cI}^{SG}(A,\rho^*) \vee AL \sqrt \frac{\log(1/\delta)}{N} \vee AL \frac{| \cO| }{N} \enspace, 
\]
where $\rho^*$ satisfying the $A,\delta$-sparsity equation with $A$ verifying assumption~\ref{assum:fast_rates_reg}.\\

The proof is split into two parts and is very similar as the one of Theorem~\ref{thm:main}. First we identify a stochastic argument holding with large probability. Then, we show on that event that $ \hat f_N^\lambda \in F \cap \big( f^* + \rho^* \bB \cap \tilde r(A,\rho^*,\delta)  \bB_2 \big) $. Then, at the very end of the proof we will control the excess risk $P\cL \hat f_N^\lambda$ where $ \hat f_N^\lambda$ is defined in equation~\eqref{def:RERM}. 

\paragraph{Stochastic arguments}
The stochastic part is the same as the one in the proof of Theorem~\ref{thm:main} where a localization with respect to the regularization norm is added. First we identifiate the stochastic event onto which the proof easily follows. Let $ \cF_{ r,\rho} = F \cap \big( f^* + \rho^* \bB \cap \tilde r(A,\rho^*,\delta)  \bB_2 \big) $ and define 

\begin{align*} 
& \Omega_{\cI} = \sup_{f \in \cF_{r,\rho}}  \bigg| \big(P-P_{\cI}  \big) \big( \ell_f - \ell_{f^*}   \big)  \bigg|  \leq c\frac{L}{\sqrt{| \cI |}} \bigg( w(\cF_{r,\rho}) + \tilde r(A,\rho^*,\delta)  \sqrt{\log(1/\delta)}  \bigg) \\
&\Omega_{\cO} = 	\sup_{f \in \cF_{r,\rho}} \bigg| \big(P-P_{\cO}  \big) |f- f^*|  \bigg|  \leq \frac{c}{\sqrt{| \cO |}} \bigg( w(\cF_{r,\rho}) + \tilde r(A,\rho^*,\delta)  \sqrt{\log(1/\delta)}  \bigg)
\end{align*}
Finally, set $\Omega = \Omega_{\cI} \cap \Omega_{\cO}$
\begin{Lemma} \label{lem:sto_arg_reg}
	Grant Assumptions~\ref{assum:distri},~\ref{assum:convex},~\ref{assum:lip_conv} and assume that $F-f^*$ is $1-$sub-Gaussian. Then the event $\Omega$ holds with probability larger than $1- \delta$
\end{Lemma}

\begin{proof}
	The proof is exactly the same as the one in the non-regularized setup where a localization with respect to the regularization norm is added. It is enough to adapt the proof with the definition of $\tilde{r}(A,\rho^*,\delta)$ from Equation~\eqref{def:comp_reg}. 
\end{proof}	

\paragraph{Deterministic argument}
In this paragraph we place ourselves on the event $\Omega$. Let us recall that for any function $f$ in $F$
\begin{equation}
P_N \cL_{f}^{\lambda} = P_N (\ell_f - \ell_{f^*}) + \lambda (\|f\| - \|f^*\| )
\end{equation}
From the definition of $\hat{f}_N^{\lambda}$, we have $P_N \cL^{\lambda}_{\hat{f}_N^{\lambda}} \leq 0$. To show that $\hat f_N^{\lambda}  \in \cF_{r,\rho}$ it is sufficient to show that for all functions $f \in F \backslash \cF_{r,\rho}$ we have $P_N \cL_f^{\lambda} >0$. Let $f$ in $f \in F \backslash \cF_{r,\rho}$. By convexity of $F$ there exist a function $f_1$ in $F$ and $\alpha \geq 1$ such that $\alpha (f_1-f^*) = f-f^*$ and $f_1 \in \partial  \cF_{r,\rho}$ where $\partial  \cF_{r,\rho}$ denotes the border of $ \cF_{r,\rho}$. Using the same convex argument as the one in the proof of Theorem~\ref{thm:main} we obtain:
\begin{equation*}
	P_N \cL_f \geq \alpha P_N \cL_{f_1} \enspace,
\end{equation*}
for $f_1 \in  \partial \cF_{r,\rho}$. Moreover, by the triangular inequality we obtain
\begin{equation*}
	\|f\| - \| f^*\| \geq \alpha (\|f_1\| - \|f^*\|  ),
\end{equation*}
and thus,
\begin{equation*}
	P_N \cL_f^{\lambda } \geq \alpha P_N \cL_{f_1}^{\lambda}
\end{equation*}
Therefore it is enough to show that $P_N \cL_{f_1}^{\lambda} >0 $ for $f_1 \in \partial \cF_{r,\rho}$. By definition of $\partial \cF_{r,\rho}$, there are two different cases 1) $f_1 \in F \cap \big(  f^* + \rho^* \bS \cap \tilde{r}(A,\rho^*,\delta) \bB_2 \big) $ and 2) $f_1 \in F \cap \big(  f^* + \rho^* \bB \cap \tilde{r}(A,\rho^*,\delta) \bS_2 \big) $. In 1) the sparsity equation will help us to show that $P_N \cL_{f_1}^{\lambda} >0$ while in 2), it will be the local Bernstein condition.\\

 Let us begin by the case 1). Let $f_1 \in F \cap \big(  f^* + \rho^* \bS \cap \tilde{r}(A,\rho^*,\delta) \bB_2 \big) $.

\begin{align*}
	P_N \cL_{f_1} = \frac{|\cI|}{N} P_{\cI} \cL_{f_1} + \frac{|\cO|}{N} P_{\cO} \cL_{f_1} \geq \frac{|\cI|}{N} P_{\cI} \cL_{f_1} -L \frac{|\cO|}{N} P_{\cO} |f_1-f^*|
\end{align*}
On $\Omega_{\cI}$ if holds that 
\begin{align}  \label{eq1_reg}
	P_{\cI} \cL_{f_1} \geq \frac{| \cI |}{N} \bigg[  -  c\frac{L}{\sqrt{ | \cI| }} \bigg( w \big( \cF_{r,\rho}  \big)  + \tilde r(A,\rho^*, \delta) \sqrt{\log(1/ \delta)} \bigg)   \bigg] \geq - c\frac{ \tilde  r ^2(A,\rho^*,\delta)}{A} \enspace,
\end{align}
where we used the definitions of $\tilde r_{\cI}^B(A,\rho^*)$ and $\tilde r(A,\rho^*,\delta)$. \\
On $\Omega_{\cO}$, it holds that
\begin{align*} 
-L \frac{|\cO|}{N} P_{\cO} |f_1-f^*|  & \geq -L \frac{|\cO|}{N}   \bigg[ \|f_1-f^*\|_{L_1(\mu)}  -  \frac{c}{\sqrt{ | \cO| }} \bigg( w \big( \cF_{r,\rho}  \big)  + \tilde r(A,\rho^*, \delta) \sqrt{\log(1/ \delta)} \bigg)   \bigg] \\
& \geq -L \frac{| \cO |}{N} \tilde r(A,\rho,\delta) - c\frac{\tilde r^2(A,\rho,\delta)}{A}  \\
& \geq - c\frac{\tilde r^2(A,\rho,\delta)}{A} \enspace,
\end{align*}	
where we also used the definitions of $\tilde r_{\cI}^B(A,\rho^*)$ and $\tilde r(A,\rho^*,\delta)$ and the fact that $| \cO | \leq |\cI|$. Consequently, we get
\begin{align*}
	P_N \cL_{f_1}  \geq - c \frac{\tilde r^2(A,\rho,\delta)}{A} \enspace,
\end{align*}
where the constant $c>0$ is chosen such that $c < 7/17$. \\

Let us turn to the control of $\lambda (\|f_1\| - \|f^*\|)$. Recall that we are in the case where $\norm{f_1-f^*}=\rho^*$ and $\norm{f_1-f^*}_{L_2}\leq \tilde r(A, \rho^*,\delta)$. Let $v\in E$ be such that $\norm{f^*-v}\leq \rho^*/20$ and $g\in \partial (\norm{\cdot})_v$. We have
\begin{align*}
	&\norm{f_1}-\norm{f^*} \geq \norm{f_1} - \norm{v} - \norm{f^*-v}\geq \inr{g,f_1-v} - \norm{f^*-v}\\
	&\geqslant\inr{g,f_1-f^*} - 2\norm{f^*-v}\geqslant\inr{g,f_1-f^*} - \rho^*/10\enspace.
\end{align*}
As the latter result holds for all $v\in f^*+(\rho^*/20)\bB$ and $g\in \partial \norm{\cdot}(v)$, since $f_1 \in \cF_{r,\rho}$, we get
\begin{equation*}
	\norm{f_1}-\norm{f^*}\geq \Delta(\rho^*) - \rho^*/10\geq 7\rho^*/10\enspace,
\end{equation*}
where the last inequality holds because $\rho^*$ satisfies the sparsity equation. Finally we obtain
\begin{align} \label{eq_first_case_reg}
	P_N \cL_{f_1}^{\lambda } \geq - c\frac{\tilde r^2(A,\rho,\delta)}{A} - \frac{7}{10}\lambda \rho^* > 0 \enspace,
\end{align}
when $\lambda \geq (10c/7) (\tilde r^2(A,\rho,\delta)/(A\rho^*))$.\\

Let us turn to the second case 2). Let $f_1 \in F \cap \big(  f^* + \rho^* \bB \cap \tilde{r}(A,\rho^*,\delta) \bS_2 \big) $. With the same analysis as 1), on $\Omega$, from Assumption~\ref{assum:fast_rates_reg} it follows that
\begin{align*}
	P_N \cL_{f_1} \geq  \frac{\tilde r^2(A,\rho^*,\delta)}{A} - c \frac{\tilde r^2(A,\rho^*,\delta)}{A}  \enspace,
\end{align*}
where the constant $c$ is the same as the one appearing in Equation~\eqref{eq_first_case_reg}.
As $\|f_1\| - \|f^*\| \geq - \|f_1-f^*\| \geq - \rho^*$, it follows that
\begin{align*}
	P_N \cL_{f_1}^{\lambda} \geq  (1-c)  \frac{\tilde r^2(A,\rho^*,\delta)}{A} - \lambda \rho^*  > 0  \enspace, 
\end{align*}
when $\lambda \leq (1-c)  (\tilde r^2(A,\rho^*,\delta)/(A\rho^*))$. Note that the condition $ c < 7/17 $ implies that such a $\lambda > 0$ exists.\\

We finish the proof by establishing the result for the excess risk. Since $\hat f_N^{\lambda} \in F \cap \big(  f^* + \rho^* \bB \cap \tilde{r}(A,\rho^*,\delta) \bB_2 \big) $, on $\Omega$
\begin{align*}
	P\cL_{\hat{f}_N^{\lambda}} \leq P_{\cI}  \cL_{\hat{f}_N^{\lambda}} + c\frac{\tilde r^2(A,\rho^*,\delta)}{A}
\end{align*}
Moreover we have
\begin{align*}
	P_{\cI}  \cL_{\hat{f}_N^{\lambda}} & = \frac{N}{|\cI|} P_N \cL_{\hat{f}_N^{\lambda}} - \frac{|\cO|}{|\cI|} P_{\cO} \cL_{\hat{f}_N^{\lambda}} = \frac{N}{|\cI|} P_N\cL_{\hat{f}_N^{\lambda}}^{\lambda }  + \lambda\frac{N}{|\cI|}(\|f^*\|- \| \hat{f}_N^{\lambda}\|  ) - \frac{|\cO|}{|\cI|} P_{\cO} \cL_{\hat{f}_N^{\lambda}}  \\
	& \leq  2\lambda \rho^*+ L\frac{|\cO|}{|\cI|} P_{\cO} |\hat{f}_N^{\lambda} - f^*|   \\
	& \leq  2\lambda \rho^*+  c \frac{\tilde r^2(A,\rho^*,\delta)}{A}  \\
	& = c  \frac{\tilde r^2(A,\rho^*,\delta)}{A}\enspace,
\end{align*}
which concludes the proof for the excess risk. 

\subsection{Proof Theorem~\ref{thm:reg} in the local bounded setting.}
The proof consists in taking the stochastic argument from the proof of Theorem~\ref{thm:main}  (and adding the localization with respect to the regularization norm) and the deterministic argument from the proof of Theorem~\ref{thm:reg}.

\begin{footnotesize}
	\bibliographystyle{plain}
	\bibliography{biblio_reg}
\end{footnotesize}

\end{document}